\documentclass[11pt]{article}
\usepackage{definitions_arxiv}

\begin{document}
	
	\title{Analysis of Dual-Based PID Controllers through Convolutional Mirror Descent}
	
	\author{Santiago R. Balseiro\thanks{Columbia Business School and Google Research (srb2155@columbia.edu).} \and Haihao Lu\thanks{The University of Chicago, Booth School of Business (haihao.lu@chicagobooth.edu). Part of the work was done at Google Research.} \and Vahab Mirrokni\thanks{Google Research (mirrokni@google.com)}\and Balasubramanian Sivan\thanks{Google Research (balusivan@google.com)}}
	
	\maketitle
	
	\begin{abstract}
	Dual-based proportional-integral-derivative (PID) controllers are often employed in practice to solve online allocation problems with global constraints, such as budget pacing in online advertising. However, controllers are used in a heuristic fashion and come with no provable guarantees on their performance. This paper provides the first regret bounds on the performance of  dual-based PID controllers for online allocation problems. We do so by first establishing a fundamental connection between dual-based PID controllers and a new first-order algorithm for online convex optimization called \emph{Convolutional Mirror Descent} (CMD), which updates iterates based on a weighted moving average of past gradients. CMD recovers, in a special case,  online mirror descent with momentum and optimistic mirror descent. We establish sufficient conditions under which CMD attains low regret for general online convex optimization problems with adversarial inputs. We leverage this new result to give the first regret bound for dual-based PID controllers for online allocation problems. As a byproduct of our proofs, we provide the first regret bound for CMD for non-smooth convex optimization, which might be of independent interest.
	\end{abstract}

\onehalfspacing

\section{Introduction}

Dual-based PID controllers are broadly used in practice to solve online optimization problems with global constraints. To concretely illustrate their use, consider a central challenge that large internet advertising platforms solve every day: pacing the budgets of advertisers' campaigns in repeated auctions. Advertisers run daily/weekly/monthly campaigns and specify budgets, which are hard constraints on the maximum amount that they are willing to pay during their campaigns. A budget pacing system bids in repeated auctions on behalf of the advertisers with the goal of maximizing the advertisers' sum of utilities subject to the budget constraint. Pacing budgets is a key problem and has seen a lot of activity in the past few years~\citep{zhang2016feedback,BalseiroGur2019MS, conitzer2021multiplicative, tashman2020dynamic,  ye2020cold}. The key challenge in a pacing system is that the budget is shared across all the auctions advertisers participate during their campaign. Hence bidding high in the current auction comes with an opportunity cost of not being able to capitalize on valuable future opportunities. 


In practice, advertisers' budgets are often paced using PID controllers. Consider an advertiser who participates in $T$ auctions during their campaign with the aim of maximizing their cumulative utilities subject to a total budget of $B$. Suppose the advertiser's valuation for each opportunity and the competition in the auction for these $T$ rounds are stationary, e.g., drawn independently from a fixed distribution. In this case, to avoid bidding too high and missing future opportunities, the advertiser should seek to spend, on average, an amount $B/T$ per round. A PID controller in such a setting would iteratively update a control variable as a function of the “error” term, namely, the difference between the ideal target spend $B/T$ and the actual spend. Typically the iterative update has terms proportional to the error, to the integral (or the summation) over the past errors, and to the derivative (or finite difference) of the error. 

More generally, PID controllers are one of the most common solutions to control problems in practice~\citep{aastrom2006advanced}, and they are used extensively for resource allocation problems. In a general online allocation problem, there are different types of resources (as opposed to the single resource in the budget pacing example), and requests are to be allocated subject to global resource constraints. As requests arrive sequentially, the decision maker is disclosed by the functions that map actions to rewards and resource consumptions and, after observing these functions, takes an action that results in a certain amount of reward and a certain amount of resource consumption for every resource. These reward and resource-consumption functions could be potentially non-convex, and, importantly, the characteristics of future requests are uncertain. Our analysis in this work applies to this general class of problems. We formally introduce online allocation problems in Section~\ref{sec:online-allocation} and the special case of budget pacing is expanded further in Section~\ref{sec:BudgetExample}.

\emph{Theory versus practice gap, and our central question.} PID controllers are extremely simple and, due to their flexibility and adaptability, extensively used in practice. Despite their wide adoption, however, these controllers come with no theoretical guarantees for online allocation problems---not even in the single resource example of budget pacing. So far, they have primarily been used in a heuristic fashion in bidding and budget pacing, among others (see, e.g., \citealt{tashman2020dynamic,zhang2016feedback,smirnov2016online,yang2019bid,ye2020cold}). The goal of this paper is to bridge the gap between theory and practice by providing a positive answer to the following question: \emph{Can we analyze PID controllers and establish provable regret guarantees on their performance?} 

\emph{What are dual-based PID controllers and why are these worth studying?} A controller is a simple feedback loop that iteratively updates control variables to reduce steady-state errors. As mentioned, in the context of online allocation, it is natural to set the error to be the difference between the ideal and actual resource consumption. But what should we pick as control variables? A natural candidate for the control variables are the dual variables corresponding to the resource constraints. The motivation is as follows. Online allocation problems can be efficiently solved using dual-based approaches that relax the resource constraints and penalize their violations by moving them to objective using dual variables. These dual variables act as the opportunity cost of consuming a unit of resource. The benefit of these dual-based approaches is that they allow one to decompose a complex problem with global constraints into simpler sub-problems that are no longer constrained. Dual-based allocation methods have been studied extensively in the past decades and have proven to perform well, both in theory and practice.

The main challenge, however, is that the optimal dual variables are often unavailable ahead of time when making online decisions. Therefore, a natural approach is to guess initial values for the dual variables and employ a PID controller on top of them to guide future allocations by controlling their evolution. When the error is positive (i.e., the actual resource consumption in a period is smaller than the ideal one), the controller would reduce dual variables, which has the effect of reducing the opportunity cost of future opportunities and decreasing the error rate by increasing future resource consumption. Figure~\ref{fig:PIcontroler}  gives a block diagram representation of a typical controller. A key open question so far has been to establish performance guarantees for this natural algorithm that performs controller updates on dual variables. 


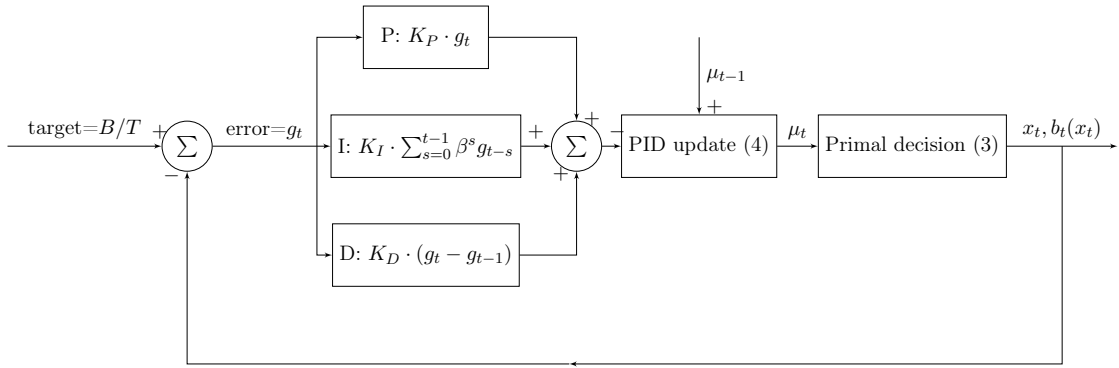
\begin{figure*}[!htb]
    \centering
    \resizebox{0.9\textwidth}{!}{%
    \begin{tikzpicture}[auto, node distance=2cm,>=latex']

        \node [input, name=input] {};
        \node [sum, right of=input,xshift=6em] (sum) {$\sum$};
        \coordinate [right of = sum,xshift=1em] (empty);
        \node [block, right of = empty] (Icontroller) {I: $K_I \cdot \sum_{s=0}^{t-1} \beta^s g_{t-s}$};
        \node [block, below of=Icontroller] (Dcontroller) {D: $K_D \cdot (g_t - g_{t-1})$};
        \node [block, above of=Icontroller] (Pcontroller) {P: $K_P \cdot g_t$};
        \node [sum, right of=Icontroller, 
                node distance=2cm,xshift=2em] (sum2) {$\sum$};
        \draw [->] (Pcontroller) -| node[pos=0.99] {$+$}  node[name=u] {} (sum2);
        \draw [->] (Icontroller) --  node {$+$} (sum2);
        \draw [->] (Dcontroller) -|  node[pos=0.99] {$+$}  node[name=u] {} (sum2);
        \node [block, right of=sum2, node distance=2.25cm] (sum3) {PID update \eqref{eq:non-linear-update}};
        \coordinate [above of=sum3] (mu_pre) {};
        \node [block, right of=sum3, xshift=5em, node distance=2cm] (system) {Primal decision \eqref{eq:primal_decision}};
        \node [output, right of=system, xshift=2em, node distance=3cm] (output) {};
        \coordinate [below of=u] (measurements);
    
        \draw [draw,->] (input) -- node[pos=0.99] {$+$} node [align=center] {\text{target}=$B/T$} (sum);
        \draw [-] (sum) -- node  [align=center] {\text{error}=$g_t$} (empty);
        \draw [->] (empty) |- node {} (Dcontroller);
        \draw [->] (empty) |- node {} (Icontroller);
        \draw [->] (empty) |- node {} (Pcontroller);
        \draw [->] (sum2) --  node[pos=0.8] {$-$} node {} (sum3);
        \draw [->] (sum3) -- node {$\mu_t$} node {} (system);
        \draw [->] (mu_pre) -- node {$\mu_{t-1}$} node[pos=0.9] {$+$} node {} (sum3);
        \draw [->] (system) -- node [name=y] {$x_t, b_t(x_t)$}(output);
        \draw [->] (y) |- (measurements);
        \draw [->] (measurements) -| node[pos=0.99] {$-$} node [near end] {} (sum);
    \end{tikzpicture}
    }
    \caption{Block diagram representation of the dual-based PID controller.}
    \label{fig:PIcontroler}
\end{figure*}

\subsection{Main Contributions}

Our main contribution is to analyze dual-based PID controllers and establish near-optimal regret guarantees for online resource allocation problems. We achieve this in three steps (steps 1 and 2 are our new contribution):

\begin{enumerate}
\item We make a formal connection between dual-based PID controllers and first-order methods for online allocation problems. In particular, we establish an equivalence between dual-based PID controllers and a new first-order algorithm for online optimization, which we call Convolutional Mirror Descent (CMD). CMD updates the iterates using a discrete convolution of past gradients with a fixed sequence $\{\lambda_i\}_{i\ge0}$, i.e., it gives weight $\lambda_i$ to the $i$-th past gradient. Any PID controller can be implemented by suitably picking the weights $\lambda_i$ of CMD. As we discuss shortly, this connection is of independent interest given the wide adoption of PID controllers and uncovers a link between optimal control theory and online algorithms.

\item We give the first regret analysis of CMD for online convex optimization, which is a powerful and widely adopted framework for online decision making~\citep{hazan2016introduction}. To the best of our knowledge, ours is the first paper to consider CMD for online convex optimization. We prove regret bounds for the case of adversarial inputs, establishing the robustness of this algorithm. Our regret bounds are parametric and isolate a few key metrics of the weights $\lambda_i$ that affect performance. In particular, CMD allows to recover mirror descent with momentum and optimistic mirror descent as special cases.


\item Using the machinery from~\cite{balseiro2020best}, we make the connection between online resource allocation and online convex optimization, and close the loop by establishing the first regret guarantees for dual-based PID controllers. {Numerical experiments on online linear programs support our theoretical findings.}
\end{enumerate}

\paragraph{The PID Controller-Mirror Descent connection.} The crux of the formal equivalence between dual-based PID controllers and dual-based first-order methods for online convex optimization is as follows. PID controllers update their control variable based on an error term that measures the difference between a target resource consumption (often called reference resource consumption) and the actual resource consumption. Online mirror descent updates the dual variable based on the subgradient of the dual function. The connection comes from the fact that the error term in a dual-based PID controller is \emph{identical} to the subgradient of the dual function in online mirror descent. This connection acts as the foundation of the equivalence (our first contribution). The connection is formally developed in Section~\ref{sec:connection} and guarantees that any analysis of CMD immediately yields guarantees for the dual-based PID controller. 

\emph{Momentum and the I-term.} The usual I-term in a PID controller adds all previous errors. It is straightforward to establish that such a crude I-term introduces undamped oscillations, and, therefore, will never converge. A natural way to control such oscillations is to consider an exponentially weighted average of past subgradients. This idea has received a lot of attention in the online optimization literature and is equivalent to \emph{momentum}. Momentum dates back to \cite{polyak1964some} and is known to provide faster convergence rates in deterministic problems. Momentum-based methods, in particular, have been used extensively in optimization and machine learning~\citep{kingma2015adam,sutskever2013importance,nesterov1983method}. The usual definition of momentum is to add a difference between two successive iterates during the update---this corresponds to an exponentially decaying sum over all the gradients, and this is the interpretation of the I-term we adopt in this paper. 

\emph{Optimism and the D-term.} The D-term in a PID controller considers the difference between two consecutive errors. The D-term can be related to optimistic gradient descent, in which the difference between two consecutive gradients is added on top of a gradient step. In machine learning, optimistic gradient descent was proposed and used for minimax optimization and training GANs ~\citep{daskalakis2017training,daskalakis2018limit}, where it has been shown to improve the stability of gradient descent~\citep{daskalakis2017training,mokhtari2020unified}.


\emph{Non-linear controllers and mirror descent.} In many applications, the control variables are adjusted multiplicatively instead of additively as additive updates tend to be too sensitive to the tuning parameters. Moreover, many controllers used in practice do not act directly over dual variables but instead control some seemingly unrelated, ad-hoc variables (see, e.g., \citealt{zhang2016feedback, conitzer2021multiplicative, tashman2020dynamic,  ye2020cold}). To accommodate such controllers, we introduce non-linear response functions to transform the control variables. While standard additive dual-based P-controllers are equivalent to dual subgradrient descent with momentum, we show that general non-linear PID controllers can be represented using CMD. The flexibility of mirror descent allows us to prove performance bounds for general non-linear controllers and, at the same time, show that many ad-hoc controllers are, in fact, dual-based controllers with appropriate non-linear response functions.

\paragraph{Key Technical Challenges.} We now briefly discuss the proof ideas/challenges for establishing the regret bound for generalized CMD. The convolutional operator makes the regret analysis more challenging than in the case of standard online mirror descent because one cannot directly apply the three-point property on the regret to construct the telescoping. One of the key elements of the proof is to decompose the regret of the algorithm against the best static action in hindsight into two distinct terms. The first term depends on the difference between the current iterate and the best static action in hindsight and can be bounded using the three-point property of convex optimization. The second term is unique to CMD as it depends on the difference of the iterates. We bound this term by proving a stability property for CMD. 

We then use our general bound for CMD to provide the first regret bounds for different types of PID controllers. PID controllers can be interpreted as linear filters and, thus, can be naturally analyzed using the theory of linear-time-invariant systems. In particular, we analyze PID controllers by studying their Z-transforms in the frequency domain instead of the time domain. This leads to a novel analysis of the convolutional operators induced by different PID controllers. As a byproduct, we establish here the first regret bound for mirror descent with momentum and optimistic mirror descent for non-smooth convex optimization. 

In the presence of a derivative and integral controller, the linear filter behaves like a harmonic oscillator whose behavior is dictated by the roots of the transfer function. We remark that our analysis is restricted to parameter values for which the transfer function has real roots, which corresponds to the linear filter being an overdamped oscillator (Assumption~\ref{ass:real-roots}). We conjecture that our analysis extends to the case of imaginary roots, which corresponds to the case of the linear filter being an underdamped oscillator. Our numerical results support this conjecture.

{
Our numerical experiments suggest that PID controllers can lead to improvements in performance over a pure P controller. In particular, the highest performance is achieved with a PI controller, thus confirming the benefits of adding momentum. We remark, however, that our strongest theoretical bound is provided for the case of a pure P controller. This is aligned with literature on online convex optimization where highly empirically successful methods do not necessarily have regret bounds better than vanilla subgradient descent (see, e.g., the bounds for ADAM in \citealt{kingma2015adam}).}

\subsection{Related Literature}\label{sec:related}

In this section, we overview the streams of literature that are most closely related and position our work vis-\'a-vis.

\paragraph{PID controllers.} PID controllers are perhaps the most used solutions to practical control problems, and have been implemented in different application domain~\citep{aastrom2006advanced,cominos2002pid,bennett2001past,vinagre2007fractional,panda2012introduction}. In particular, in the context of online allocation problems, PID controllers have been recently employed to handle budget pacing problems in internet advertising markets~\citep{yang2019bid,zhang2016feedback,tashman2020dynamic,ye2020cold}. { Most analyses of PID controllers are focused on linear systems \citep{aastrom2006advanced} and, to the best of our knowledge, ours is the first paper to provide regret bounds for dual-based PI controllers in the context of online allocation problems.

\paragraph{First-order methods for online convex optimization.} Online convex optimization (OCO) studies online decision making with convex loss functions. Some classic online first-order methods include online gradient descent, online mirror descent, regularized-follow-the-leader, and online adaptive methods among others~\citep{hazan2016introduction}. The major difference between online optimization compared to stochastic optimization is that the loss functions can be chosen adversarially in online optimization, while the loss functions in stochastic optimization come from an unknown stochastic i.i.d.~distribution.


In the last decade, there has been a renewed interest in first-order methods as these have been particularly successful in solving challenging optimization problems such as training machine learning prediction methods~\citep{nesterovBook,sra2012optimization,beck2017first} and online allocation problems such as budget pacing in ad markets~\citep{Devanur2019near,gupta2016experts,BalseiroGur2017EC,conitzer2021multiplicative}. 

Online algorithms have achieved great success in machine learning training. In particular, the adaptive gradient methods such as AdaGrad~\citep{duchi2011adaptive} adaptively scale the step-size for each dimension and can speed up the training of many machine learning problems. Later on, \cite{kingma2015adam} proposes ADAM, and its major novelty is to apply the exponential weighted average to gradient estimates, which is equivalent to momentum method for unconstrained problems. While ADAM is one of the most successful algorithms for training neural networks nowadays, \cite{reddi2019convergence} shows that the theoretical analysis of ADAM contains major technical issues and indeed ADAM may not even converge. After that, many variants of ADAM have been proposed~\citep{reddi2019convergence,chen2018closing,huang2018nostalgic}, but most of the analysis requires the momentum parameter $\beta\rightarrow 0$. More recently, \cite{alacaoglu2020new} proposed a simple analysis for variants of ADAM that allows arbitrary momentum parameter $\beta$, but the regret bound is $1/(1-\beta)$ times larger than vanilla OGD.

We contribute to the OCO literature by introducing Convolutional Mirror Descent, a first-order method for online convex optimization, which generalizes mirror descent with momentum and mirror descent with momentum, and providing a regret bound on its performance.
}

\paragraph{Online Allocation.} Online allocation is an important problem in computer science and operations research with wide applications in practice. Many previous works of online allocation focus on the stochastic i.i.d. input model. In particular, \cite{DevanurHayes2009} present a two-phase dual training algorithm for the AdWords problem (a special case of online allocation problem): estimating the dual variables by solving a linear program in the first exploration phase; and taking actions using the estimated dual variables in the second exploitation phase. They show that the proposed algorithm obtains regret of order $O(T^{2/3})$. Soon after, \cite{Feldman2010} present similar two-phases algorithms for more general linear online allocation problems with similar regret guarantees. Later on, \cite{Agrawal2014OR}, \cite{Devanur2019near} and \cite{kesselheim2014primal} propose primal- and/or dual-based algorithms that dynamically update decisions by periodically solving a linear program using all data collected so far, and show that the proposed algorithms have $O(T^{1/2})$ regret, which turns out to be the optimal regret rate in $T$. More recently,~\cite{balseiro2020dual,balseiro2020best,li2020simple} propose simple dual descent algorithms for online allocation problems with stochastic inputs, which attain $O(T^{1/2})$ regret. The algorithm updates dual variables in each period in linear time and avoids solving large auxiliary programs. Dual-based PID controllers fall into this category: the update per iteration can be efficiently computed and there is no need to solve large convex optimization problems.

\section{Dual-Based PID Controller for Online Allocation Problems}

In this section, we introduce the online allocation problem following the standard notation in \cite{balseiro2020best} and a natural dual-based PID controller that is used in practice to solve such problems. We then establish a fundamental connection between dual-based PID controllers and first-order methods for online convex optimization.

\subsection{Online Allocation Problems} \label{sec:online-allocation}
We here consider an online optimization problem with a finite horizon of $T$ time periods and resource constraints. At time $t$, the decision maker receives a request $\gamma_t = (f_t,b_t,\cX_t) \in \mathcal S$ where $f_t: \cX_t\rightarrow\RR_+$ is a non-negative reward function, $b_t:\cX_t\rightarrow\RR_+^m$ is a non-negative resource consumption function, and $\cX_t \subset \RR^d_+$ is a compact set. We denote by $\mathcal S$ the set of all possible requests that can be received. After observing the request, the decision maker takes an action $x_t\in \cX_t\subseteq \RR^d$ that leads to reward $f_t(x_t)$ and consumes $b_t(x_t)$ resources. The total amount of resources is $B\in \RR^m_{+}$ with $B_j > 0$ for all $j$. We denote by $\vgamma = (\gamma_1,\ldots,\gamma_T)$ the vector of inputs  over time $1,\ldots,T$.

{Online allocation problems have a plethora of applications in practice, including capacity allocation problems in airline and hotel revenue management~\citep{bitran2003overview,talluri2006theory,gallego2019revenue}, allocation of advertising opportunities in internet advertising markets~\citep{karp1990optimal,DevanurHayes2009,feldman2009online,Feldman2010}, bidding in repeated auctions with budgets~\citep{BalseiroGur2019MS}, personalized assortment optimization in online retailing marketplaces~\citep{bernstein2015dynamic, golrezaei2014real}, etc.}

Suppose all future information is known to the user, we can formalize the offline problem as
\begin{align}\label{eq:OPT}
\begin{split}
    \OPT(\vec \gamma)=
	\sup_{\vec x: x_t\in \cX_t} & \sum_{t=1}^T f_t(x_t)
	\text{ s.t. }       \sum_{t=1}^T b_t(x_t) \le B\,.
\end{split}
\end{align}
The offline optimum provides an upper bound on reward of any online algorithm as it captures the best possible performance under full information of all requests. Thus, we can benchmark any online algorithm against the reward of the optimal solution when the request sequence $\vgamma$ is known in advance, i.e., the offline optimum. 

\subsection{Dual-Based PID Controller for Online Allocation}

Dual-based algorithms provide a natural approach to solving online optimization problems with global constraints. The basic idea is to consider the Lagrangian dual of \eqref{eq:OPT} by introducing the dual variable $\mu \in \mathbb R_m^+$ for the resource constraints:
\begin{align}\label{eq:Langrangian}
\begin{split}
	\OPT(\vec \gamma)&=\sup_{\vec x: x_t\in \cX_t} \inf_{\mu\ge 0} \left\{ \sum_{t=1}^T f_t(x_t) - \mu^\top      \left(\sum_{t=1}^T b_t(x_t)- B\right) \right\} \\
	& \le \inf_{\mu\ge 0} \left\{  \mu^\top B + \sum_{t=1}^T \sup_{x_t\in \cX_t} \left\{ f_t(x_t)-\mu^\top b_t(x) \right\} \right\}\,,
\end{split}
\end{align}
where the equality follows because it is optimal to let $\mu_j \uparrow \infty$ if $\sum_{t=1}^T b_{t,j}(x_t) > B_j$ and $\mu_j=0$ otherwise, and the inequality from the minimax inequality and using that the objective is separable. The key observation is that the primal variable $x_t$ is separable in the Lagrangian form \eqref{eq:Langrangian}. Suppose the optimal dual variable $\mu^*$ is known to the decision maker, then a natural algorithm to make the primal decision is by setting
$$
x_t=\arg \max_{x_t\in \cX_t} \left\{ f_t(x_t)-(\mu^*)^\top b_t(x) \right\} \ .
$$
That is, we compute the action that maximizes the Lagrangian at each step $t$, given $\mu^*$. In practice, $\mu^*$ is usually unknown to the user, and instead, we can iteratively update the dual variable $\mu_t$ using a PID controller.

\begin{algorithm}[t]
	\SetAlgoLined
	{\bf Input:} Total time periods $T$, initial resources $B_1=B$, proportional term parameter $K_P \ge 0$, integral term parameter $K_I \ge 0$, derivative gain parameter $K_D \ge 0$, exponential averaging term $\beta \in [0,1)$, initial dual solution $\mu_1 \in \mathbb R_+^m$, and increasing, continuous functions $\ell_j : \RR_+ \rightarrow \RR$. \\
	\For{$t=1,\ldots,T$}{
		Receive request $(f_t, b_t, \cX_t)$.\\
		Make the primal decision $x_t$ and update the remaining resources $B_t$:
		\begin{align}
		\tilde{x}_{t} &= \arg\max_{x\in\cX_t}\left\{f_{t}(x)-\mu_{t}^{\top} b_{t}(x)\right\} \ ,\label{eq:primal_decision} \\
		x_{t}&=\left\{\begin{array}{cl}
		\tilde{x}_t    & \text{  if } b_t(\tilde{x}_t)\le B_t  \\
		0
		& \text{ otherwise} 
		\end{array} \right. \ , \nonumber \\
		B_{t+1} &= B_t - b_t(x_t) . \nonumber
		\end{align}\\
		Determine the error by comparing the target and actual consumption:
		$ \  \tg_t = B/T -b_t( x_t)  
		\ .$\\
		
		Update the dual variables using a PID controller with non-linear response:
		{\scriptsize
		\begin{align}\label{eq:non-linear-update}
    \mu_{t+1,j} = \max\Bigg(\ell_j^{-1} \Bigg(  \ell_j\big(\mu_{t,j}\big) -  \underbrace{K_P g_t}_{\substack{\text{proportional}\\\text{term}}} - \underbrace{K_I \sum\nolimits_{s=0}^{t-1} \beta^{s} g_{t-s}}_{\text{integral term}} - \underbrace{K_D (g_t-g_{t-1})}_{\substack{\text{derivative}\\\text{term}}} \Bigg), 0\Bigg)\,,
\end{align}}
(See \eqref{eq:api} and \eqref{eq:mpi} for examples of the non-linear responses, and how they implement well-known controller variants such as additive and multiplicative.)
	 }
	\caption{Dual-Based PID Controller for Online Allocation}
	\label{al:PID}
\end{algorithm}

When resource constraints are binding, the algorithm should aim to deplete resources evenly over time to avoid depleting resources too early and missing good opportunities later in the horizon. The term $B/T$ can be interpreted as the \emph{target consumption}, i.e., how much resources should ideally be consumed on average per time period. The term $b_t(x_t)$ is the \emph{actual consumption}, i.e., how much resources are actually consumed per time period. The difference between these two terms gives the \emph{error} in resource consumption, i.e., $g_t = B/T - b_t(x_t)$. As a result, a natural algorithm that is widely used in practice for online allocation problems is running a PID controller on the dual variable $\mu_t$ based on the error. That is, the controller adjusts the control variables $\mu_t$ to make the error $g_t$ as small as possible. The controller has three terms. The P term is proportional to the current value of error: if the error is positive, we need to reduce the control to decrease the error. The I term takes an exponential average of all past errors, while the D term takes the difference between the two last errors. The PID controller has a coefficient (a.k.a.~gain factor) $K_P$ for the proportional term, a coefficient $K_I$ for the integral term, and a coefficient $K_D$ for the derivative term. These coefficients jointly determine how the errors impact the change in the control variable. The dual-based PID controller algorithm is stated in Algorithm \ref{al:PID} and a block diagram representation is provided in Figure~\ref{fig:PIcontroler}.  Three observations are in order.

Firstly, in many circumstances, it is more appropriate to update the control variables non-linearly. For these reasons, the PID controller takes as input increasing \emph{response} functions $\ell_j: \RR_+ \rightarrow \RR$ for each resource and runs the PID controller on the transformed control variables $\ell_j\big(\mu_{t,j}\big)$, which yields the update rule in \eqref{eq:non-linear-update}. In particular, when the response function is chosen to be the identity function $\ell_j(\mu_j)=\mu_j$, we recover the additive PID controller:
\begin{align}\label{eq:api}
    \mu_{t+1} = \max\Bigg(  \mu_t - \underbrace{K_P g_t}_{\substack{\text{proportional}\\\text{term}}} - \underbrace{K_I \sum\nolimits_{s=0}^{t-1} \beta^{s} g_{t-s}}_{\text{integral term}} - \underbrace{K_D (g_t-g_{t-1})}_{\substack{\text{derivative}\\\text{term}}} , 0\Bigg)\,.
\end{align}
When the response function is chosen to be the logarithm $\ell_j(\mu_j)=\log \mu_j$, we recover the multiplicative PID controller:
\begin{align}\label{eq:mpi}
    \mu_{t+1} = \mu_t \circ \exp\Bigg(-   \underbrace{K_P g_t}_{\substack{\text{proportional}\\\text{term}}} - \underbrace{K_I \sum\nolimits_{s=0}^{t-1} \beta^{s} g_{t-s}}_{\text{integral term}} - \underbrace{K_D (g_t-g_{t-1})}_{\substack{\text{derivative}\\\text{term}}}\Bigg)\, ,
\end{align}
where $\circ$ is the coordinate-wise product.

Secondly, ours is an \emph{incremental} or velocity controller as it adjusts the rate of change of the control variable as opposed to positional controllers, which act directly on the control variable (see, e.g, \citealt[chapter 13.4]{aastrom2006advanced}). Incremental controllers are used in applications, such as motors, in which it is necessary to control the velocity instead of the position of a system. Incremental controllers are natural for online allocation problems because we seek to control the velocity at which resources are consumed, i.e., the rate of consumption. We remark that adjusting the rate of change is common in practice, so we adopt this interpretation (see, e.g., \citealt{zhang2016feedback,smirnov2016online,conitzer2021multiplicative,BalseiroGur2019MS,ye2020cold}). As a result, the P term in our incremental controller acts as an I term in a positional controller, which guarantees that the steady-state error becomes close to zero. 

Thirdly, we utilize an exponential averaging in the integral term. It recovers the traditional integral controller when choosing $\beta=1$. However, one can observe that in the case $\beta=1$, the integral term introduces undamped oscillation of the dual variable $\mu_t$, which can push $\mu_t$ far away from the optimal dual variable $\mu^*$. The exponential average puts more weight on recent error terms, which reduces oscillations over time. In the rest of the paper, we focus on the case when $\beta<1$.


\subsection{Example: Pacing Budgets in Repeated Auctions}\label{sec:BudgetExample}


Consider an advertiser with a budget of $B$ dollars who participates in $T$ repeated auctions. The advertiser is a quasi-linear utility maximizer with a budget constraint that limits the total expenditure over the $T$ auctions. Each request constitutes an auction in which the advertiser can bid to show an ad. In each auction, the advertiser is disclosed some attributes about the advertising opportunity and develops an estimate $v_t$ for the value of showing an ad. Based on this estimate, the advertiser needs to decide on a bid $x_t$. We denote by $f_t(x_t)$ the utility the advertiser derives when the bid is $x_t$ and $b_t(x_t)$ the payment made to the auctioneer. Because the advertiser's problem has a single resource constraint, the algorithm needs to maintain a single parameter: the dual variable $\mu_t$ of the budget constraint.

To simplify the exposition, we focus on second-price auctions, but our analysis extends to other truthful auctions. In a second-price auction, the highest bidder wins, but the winner pays the second-highest bid in the auction. In particular, denoting by $d_t$ the highest competing bid in the auction, the utility of the advertiser is $f_t(x_t) = (v_t - d_t) \mathbf 1\{ x_t \ge d_t\}$ and $b_t(x_t) = d_t \mathbf 1\{ x_t \ge d_t\}$. Notice that the feedback structure does not exactly match that of \eqref{sec:online-allocation} because the advertiser does not know the highest competing bid $d_t$ at the point of bidding. Nevertheless, as showed by \citet{BalseiroGur2019MS}, an optimal bid when the dual variable is $\mu_t$ can be computed as follows:
\begin{equation}\label{eq:bidding-update}
\begin{split}
		\tilde{x}_{t} &= \arg\max_{x\ge0}\left\{f_{t}(x)-\mu_{t}^{\top} b_{t}(x)\right\} =\arg\max_{x\ge0}\left\{ (v_t - (1+\mu_t) d_t) \mathbf 1\{ x_t \ge d_t\} \right\} = \frac {v_t} {1+\mu_t}\,,
\end{split}
\end{equation}
where we used that the second argmax expression is equivalent to that of bidding in an auction with value $v_t / (1+\mu_t)$, and the truthfulness of a second-price auction. In such an auction, the optimal bid is $v_t / (1+\mu_t)$. The dual variable can then be updated additively or multiplicatively using \eqref{eq:api} or \eqref{eq:mpi}, respectively.

The bidding strategy \eqref{eq:bidding-update} has been studied in various papers, including \citet{gummadi2012repeated,BalseiroGur2019MS,yang2019bid}. In many settings, however, bids are computed by multiplying values by a variable factor. The following multiplicative bidding strategy is considered in \cite{zhang2016feedback, conitzer2021multiplicative, tashman2020dynamic,  ye2020cold}:
\begin{equation}\label{eq:alternative-pacing}
    \tilde x_t = \nu_t^{1/q} \cdot v_t\,,
\end{equation}
where $\nu_t \in (0,1]$ is a shading factor and $q \ge 1$ is a fixed power. At first sight, the previous update rule does not seem to conform with the dual-based PID controller in Algorithm~\ref{al:PID}. However, we can leverage the nonlinear response function to reverse engineer an update that yields the desired bidding strategy. Assume, for example, that the shading factor $\nu_t$ is updated multiplicatively using the update 
$$
    \nu_{t+1} = \nu_t \cdot \exp\left(K_P g_t + K_I \sum_{s=0}^{t-1} \beta^{s} g_{t-s}+K_D(g_t - g_{t-1}) \right)\,.
$$
(We changed the sign of the errors because the bid is now increasing in the control variable $\nu_t$.) Some calculus shows that the response function $\ell(s) = q \cdot \log(1+s) - 1$ together with the change of variables $\mu = \nu^{-1/\alpha} - 1$ induces the desired bidding strategy. This illustrates the flexibility of the dual-based PID controller: by carefully designing the response function, we can apply our theory to a large scope of bidding algorithms used in practice, and not just to dual-based PID controllers.


\section{Convolutional Mirror Descent and Connection to PID Controllers}

In this section, we establish a connection between dual-based PID controllers and convolutional mirror descent, a new first-order method based on online mirror descent that takes a weighted average of past gradients. We begin by formally stating the online convex optimization problem, then we introduce convolutional mirror descent, and we conclude by establishing the aforementioned connection.

\subsection{Online Convex Optimization}

Here we consider the online convex optimization problem with $T$ time horizons over a feasible set $\mathcal U \subseteq \RR^m$, which is assumed to be convex and closed. For each time step $t$, an algorithm selects a solution $\mu_t \in \mathcal U$; following this, a convex loss function $w_t :\mathcal U \rightarrow \RR$ is revealed, and the algorithm incurs a cost $w_t(\mu_t)$. We assume that the algorithm has access to subgradients $g_t \in \partial w_t(\mu_t)$. We remark that the subgradients $g_t \in \RR^m$ are observed \emph{after} taking the action. Then, for any $\mu \in \mathcal U$, we denote the regret of the online algorithm against a fixed static action $\mu \in \mathcal U$ by
\[
    \RegretNoArg(\mu) := \sum_{t=1}^{T} w_t(\mu_t) - w_t(\mu)\,.
\]


We here focus on the \emph{adversarial model}, where the functions $w_t(\mu)$ and the $g_t$ subgradients are generated adversarially. The adversary can choose the functions adaptively and after observing the action chosen by the algorithm. We remark that our results are more general than needed to tackle online allocation problems as they handle arbitrary convex functions and convex feasible sets. As such, our analysis can be of independent interest to researchers in the online convex optimization literature.


\subsection{Convolutional Mirror Descent}

\begin{algorithm}[h]
	\SetAlgoLined
	{\bf Input:} feasible set $\mathcal U$, initial solution $\mu_1$, parameter sequence $\{\lambda_i\}_{i \ge 0}$, step-size $\eta$, convex reference function $h:\mathcal U \rightarrow \RR$.\\
	\For{$t=1,\ldots,T$}{
	    Play action $\mu_t$ and receive subgradient $g_t \in \partial w_t(\mu_t)$.\\
		Compute $z_t=\sum_{s\le t} \lambda_{t-s} g_s.$\\
		Compute $\mu_{t+1} = \min_{\mu \in \mathcal U} \left\{ z_t^\top \mu + \frac 1\eta V_h(\mu,\mu_t) \right\}$ where $V_h(x,y)=h(x)-h(y)-\nabla h(y)^\top (x-y)$ is the Bregman divergence of $h$. 
		}
	\caption{Convolutional Mirror Descent}
	\label{al:omd-m}
\end{algorithm}

Algorithm~\ref{al:omd-m} presents our main algorithm in the language of online convex optimization.
The algorithm has three parameters: a step size $\eta \ge 0$, a sequence $\{\lambda_i\}_{i \ge 0}$ with $\lambda_i \in \mathbb R$, and a reference function $h:\mathcal U \rightarrow \RR$. 
The algorithm plays the action $\mu_t$ and receives a subgradient $g_t$. The convolution of the past gradients and the sequence $\lambda_i$ is computed and then the iterate moves in the opposite direction of the convolution. Movement from the incumbent solution is penalized using the Bregman divergence  $V_h(x,y)=h(x)-h(y)-\nabla h(y)^\top (x-y)$ as a notion of distance. Furthermore, the solutions are projected to the feasible set $\mathcal U$ to guarantee the feasibility. The advantage of using the Bregman divergence in the projection step is that, in many cases, it allows one to better capture the geometry of the feasible set and, in some settings, leads to closed-form formulas for the projection. 

We conclude by discussing some instantiations of the algorithm for some common reference functions. First, when the reference function is the squared-Euclidean norm, i.e., $h(\mu)=\frac{1}{2}\|\mu\|_2^2$, the update rule can be written as $\mu_{t+1}=\text{Proj}_{\mathcal U}(\mu_t-\eta z_t)$ where $\text{Proj}_{\mathcal U}(x) = \min_{\mu \in \mathcal U} \| x - \mu \|_2$ denotes the projection of a point $x \in \mathbb R^m$ to the feasible set $\mathcal U$. Second, suppose the feasible set is the unit simplex $\mathcal U = \big\{ \mu \in \RR_+^m : \sum_{j=1}^m \mu_j = 1\big\}$, where $\RR_+$ denotes the set of non-negative real numbers.  In this case, a well-known choice for the reference function is the negative entropy $h(\mu) = \sum_{j=1}^m \mu_j \log(\mu_j)$. When the initial solution lies in the relative interior of the feasible set, the update rule is given by
\[
    \mu_{t+1,j} = \frac{\mu_{t,j} \exp\left( - \eta z_{t,j}\right)}{\sum_{i=1}^m \mu_{t,i} \exp\left( - \eta z_{t,i}\right)}\,.
\]

\subsection{Connection to Dual-Based PID Controllers}\label{sec:connection}

The dual-based PID controller (Algorithm \ref{al:PID}) can be reformulated as a dual first-order method. Notice that the Lagrangian dual problem to the offline problem is
\begin{align}\label{eq:dual}
    \min_{\mu\ge 0} D(\mu | \vgamma ):=\sum_{t=1}^T \pran{f_t^*(\mu)+(B/T)^\top \mu} \ ,
\end{align}
where $f_t^*(\mu):=\sup_{x\in \cX_t} \{f_t(x)-\mu^{\top} b_t(x)\}$ is the optimal opportunity-cost-adjusted reward of request $\gamma_t$. Denote the $t^\text{th}$ term of the dual function by
\begin{align}\label{eq:application-w}
    w_t(\mu):=f_t^*(\mu) + (B/T)^\top \mu\,,
\end{align}
then we have $D(\mu | \vgamma )=\sum_{t=1}^T w_t(\mu)$. Because $f_t^*(\mu)$ is the supremum of linear functions, it is always a convex function of the dual variable $\mu$. A key observation is that the error term used in our PID controller $g_t= B/T -b_t( x_t)$ is also a subgradient to $w_t(\mu)$ at $\mu_t$ whenever $\tilde x_t = x_t$. That is, from Danskin's Theorem we have that $g_t \in \partial w_t (\mu_t)$ and dual-based PID controllers can be interpreted as minimizing the functions $w_t(\mu)$ in \eqref{eq:application-w}. This connection acts as the foundation of the equivalence between dual-based PID controllers and first-order methods, as stated in the next proposition. The proof of Proposition~\ref{prop:PID-equi} is provided in Appendix~\ref{sec:proof:prop:PID-equi}.

\begin{prop}\label{prop:PID-equi}
The dual-based PID controller with non-linear response \eqref{eq:non-linear-update} is equivalent to convolutional mirror descent with feasible set $\mathcal U = \mathbb R_+^m$, reference function $h(\mu) = \sum_{j=1}^m \int_0^{\mu_j} \ell_j(s) ds$, and weights
\begin{align}\label{eq:new-al}
    \begin{split}
    \lambda_0&=1-\alpha_{I}+\alpha_I(1-\beta)\\
    \lambda_1&=-\alpha_D+\alpha_I(1-\beta) \beta \\
    \lambda_{i}&=\alpha_I (1-\beta)\beta^i \text{ for } i \ge 2
    \end{split}
\end{align}
with the step-size equal to $\eta=K_P+K_I/(1-\beta) + K_D \ge 0$ where $\alpha_I = K_I/(\eta (1-\beta))$, and $\alpha_D = K_D/\eta$. 
\end{prop}


To end this section, we comment on the update rule \eqref{eq:new-al}. First, in Proposition~\ref{prop:PID-equi}, we state the PID controller in terms of the parameters $\alpha_P, \alpha_I, \alpha_D \ge0$. Because  $\alpha_P+\alpha_I + \alpha_D = 1$, these parameters can be readily interpreted as the weights assigned to the proportional, integral, and derivative terms, respectively. Notice that \eqref{eq:new-al} only utilizes the subgradient information $g_t$ to make the update, thus, it is a first-order method for solving the offline dual problem \eqref{eq:dual}. 
In the proof, we show that our choice of reference function yields exactly the update~\eqref{eq:non-linear-update}.  The algorithm first takes a weighted average of the current gradient $g_t$, the exponentially smoothed average of all past gradients $e_t = (1-\beta) \sum_{s=0}^{t-1} \beta^{s} g_{t-s}$, and the difference between the last two gradients. Then, it performs an update using mirror descent. 
When $\alpha_P=1$, we recover vanilla mirror descent. Furthermore, when $\alpha_I=1$, we obtain $z_t=\beta z_{t-1} + (1-\beta)  g_t$, which can be interpreted as adding momentum or inertia to the solution and can be thought of as combining Polyak's heavy ball~\citep{polyak1964some} with mirror descent. In other words, $\alpha_I=1$ recovers online mirror descent with momentum. When $\alpha_I = 0$ and $\alpha_D \in (0,1)$, we obtain optimistic mirror descent. We provided a more detailed analysis of the algorithm in the next section.

\section{Regret Bounds for Convolutional Mirror Descent}\label{sec:regret-bound}


In this section, we first provide a regret bound for convolutional mirror descent under general weights $\lambda_i$ in the traditional online convex optimization setting. We then use this theoretical foundation to analyze dual-based PID controllers.

To analyze the regret bound of the new algorithm, we introduce the following standard assumptions on the reference function $h$ used in the mirror descent update. In the following, we denote the primal norm by $\|\cdot\primalnorm$ and its dual norm by $\|\mu\dualnorm = \max_{y : \| y \| \le 1} \mu^\top y$. 

\begin{ass}[Reference function]\label{ass:h} We assume
 \begin{enumerate}
     \item $h(\mu)$ is either differentiable or essentially smooth \citep{bauschke2001essential} in $\mathcal U$.

     \item $h(\mu)$ is $\sigma$-strongly convex in $\|\cdot\dualnorm$-norm in $\mathcal U$, i.e., $h(\mu_1)\ge h(\mu_2) + \nabla h(\mu_2)^\top (\mu_1-\mu_2) + \frac{\sigma}{2}\|\mu_1-\mu_2\dualnorm^2$ for any $\mu_1,\mu_2\in\mathcal U$.
 \end{enumerate}
 \end{ass}



The next theorem presents our main regret bound on Algorithm~\ref{al:omd-m} for online convex optimization with adversarial inputs. We make no assumptions on the input other than the fact that the subgradients are bounded. In addition, the following result holds for arbitrary primal-dual norm pairs.

Denote by $R$ a lower-triangular Toeplitz matrix with $\{\lambda_i\}_{i \ge 0}$ on its lower diagonals and by $Q=R^{-1}$ its inverse. Since $R$ is a lower-triangular Toeplitz matrix, $Q=R^{-1}$ is also a lower-triangular Toeplitz matrix. That is
\[
R= \begin{bmatrix}
\lambda_0 & 0 & \ddots & 0 & 0 \\
\lambda_1 & \lambda_0 & 0 & \ddots & 0 \\
\lambda_2 & \lambda_1 & \lambda_0 & 0 & \ddots \\
\ddots & \ddots & \ddots & \ddots & 0 \\
\lambda_{T-1} & \ddots & \lambda_2 & \lambda_1 & \lambda_0 
\end{bmatrix}
\quad \text{and} \quad  Q=R^{-1}=\begin{bmatrix}
q_0 & 0 & \ddots & 0 & 0 \\
q_1 & q_0 & 0 & \ddots & 0 \\
q_2 & q_1 & q_0 & 0 & \ddots \\
\ddots & \ddots & \ddots & \ddots & 0 \\
q_{T-1} & \ddots & q_2 & q_1 & q_0 
\end{bmatrix}\,.
\]
The next theorem provides a parametric regret bound on convolutional mirror descent that depends on the matrix $R$ and its inverse $Q$.

\begin{thm}\label{thm:adversarial}
Consider the sequence of convex functions $w_t(\mu)$. Let $g_t \in \partial_\mu w_t(\mu_t)$ be a sub-gradient and let $a_t=q_0+\cdots+q_{T-t}$. Suppose  $\| g_t \primalnorm \le G_1$, $\|z_t\primalnorm\le G_2$, $a_t \ge 0$, and the reference function satisfies Assumption~\ref{ass:h}. Then, the regret of Algorithm~\ref{al:omd-m} satisfies for all $\mu \in \mathcal U$
\begin{align}\label{eq:RFTL}
    &\RegretNoArg(\mu) := \sum_{t=1}^{T} w_t(\mu_t) - w_t(\mu)\\ 
    &\le
    \frac{\eta G_2^2}{2\sigma} \sum_{t=1}^T a_t + \frac{a_1}{\eta} V(\mu,\mu_1) +\frac{1}{\eta}\pran{\max_{1\le s\le T} V(\mu, \mu_s) \sum_{s=2}^T (a_s - a_{s-1})^+} 
    + \sqrt{2} G_1 G_2 \frac{\eta}{\sigma} \sum_{j=1}^T \sum_{k=1}^j k |a_j \lambda_k|\,.\nonumber
\end{align}
\end{thm}

{Before proving the regret bound of CMD, we discuss each term in \eqref{eq:RFTL}. The first two terms are typical of first-order methods such as gradient descent and mirror descent. The first term is proportional to the norm of the gradients and the distance traveled by the algorithm. The second term depends on the distance from the initial solution to the benchmark solution $\mu$. The third and fourth terms depend on the inverse of the weights $\{\lambda_i\}_{i \ge 0}$ and how they decay over time. These terms are zero for P controllers. The third term is zero for controllers for which the sequence $a_t$ is monotonically decreasing, such as PD controllers. For PID controllers the sequence $a_t$ is not necessarily zero nor decreasing, and the third and fourth terms are nonzero. In Section~\ref{sec:analysis-PID}, we study the sequence $a_t$ for different types of controllers and show it is possible to attain $O(T^{1/2})$ regret by choosing step-sizes of the order $\eta \sim T^{-1/2}$.}


\subsection{Proof of Theorem~\ref{thm:adversarial}}

The next three lemmas are useful for the analysis. Their proofs are available in Appendix~\ref{sec:proof_thm}. The first property is that the dual solutions produced by our algorithm are stable in the sense that the change in the multipliers from one period to the next is proportional to the step-size and the norm of $z_t$.

\begin{lem}\label{lemma:diff-iterate}
It holds for any $t\ge1$ that
$\|\mu_{t+1}-\mu_{t}\dualnorm\le \frac{\sqrt{2}}{\sigma} \eta\|z_t\primalnorm \ .$
\end{lem} 

The second lemma provides a fundamental decomposition of the regret bound in two distinct terms, which can be analyzed independently. 

\begin{lem}\label{lem:decomposition}
Let $b_{t,s}=        \sum_{j=t}^T a_j \lambda_{j-s}$. It holds that
\begin{equation}\label{eq:decomposition}
    \sum_{t=1}^T g_t^{\top}(\mu_t-\mu)= \sum_{t=1}^T a_t z_t^\top(\mu_t-\mu)-\sum_{s=1}^T \sum_{t=s+1}^T b_{t,s} g_s^{\top} (\mu_t-\mu_{t-1}) \,.
\end{equation}
\end{lem}

The third and final lemma is an extension of the three-point property in convex optimization.
\begin{lem}\label{lemma:omd} It holds for any $t\ge1$ that
\begin{equation*}  
    \langle z_t, \mu_t - \mu \rangle
   \le \frac{\eta}{2\sigma}\|z_t\primalnorm^2  + \frac{1}{\eta} V_h(\mu,\mu_t) - \frac{1}{\eta} V_h(\mu,\mu_{t+1})  \ .
\end{equation*}
\end{lem}

\begin{proof}[Proof of Theorem \ref{thm:adversarial}]

It follows from convexity of $w_t(\mu)$ and Lemma~\ref{lem:decomposition} that
    \begin{align*}
        \sum_{t=1}^{T} w_t(\mu_t) - w_t(\mu) \le \sum_{t=1}^T g_t^{\top}(\mu_t-\mu)= \underbrace{\sum_{t=1}^T a_t z_t^\top(\mu_t-\mu)}_{S_1}-\underbrace{\sum_{s=1}^T \sum_{t=s+1}^T b_{t,s} g_s^{\top} (\mu_t-\mu_{t-1})}_{S_2}\ .
    \end{align*}
We now bound each term at a time. For the second term, we have
\begin{align}\label{eq:first}
\begin{split}
    |S_2| &\le \sum_{s=1}^T \sum_{t=s+1}^T |b_{t,s} g_s^{\top} (\mu_t-\mu_{t-1})|
    \le G_1  \max_{1\le l\le T}\|\mu_l-\mu_{l-1}\dualnorm\sum_{s=1}^T \sum_{t=s+1}^T |b_{t,s}| \\
    &\le G_1  \max_{1\le l\le T}\|\mu_l-\mu_{l-1}\dualnorm\sum_{s=1}^T \sum_{t=s+1}^T \sum_{j=t}^T |a_j \lambda_{j-s}| \\
    &=G_1  \max_{1\le l\le T}\|\mu_l-\mu_{l-1}\dualnorm \sum_{j=1}^T \sum_{k=1}^j k |a_j \lambda_k| 
    \le \sqrt{2} G_1 G_2 \frac{\eta}{\sigma} \sum_{j=1}^T \sum_{k=1}^j k |a_j \lambda_k|\ ,
\end{split}
\end{align}
where the first inequality follows from the triangle inequality, the second from Cauchy-Schwartz and $\|g_s\primalnorm\le G_1$, and the last inequality uses Lemma \ref{lemma:diff-iterate}.

We now move to the first term. By Lemma~\ref{lemma:omd}, we have 
\begin{align*}
    &z_t^{\top}(\mu-\mu_t) + \frac{1}{\eta} V_h(\mu, \mu_t) \ge  z_t^{\top}(\mu_{t+1}-\mu_t) + \frac{1}{\eta} V_h(\mu, \mu_{t+1}) + \frac{1}{\eta} V_h(\mu_t, \mu_{t+1}) \\
    \ge & z_t^{\top}(\mu_{t+1}-\mu_t) + \frac{1}{\eta} V_h(\mu, \mu_{t+1}) + \frac{\sigma}{2 \eta}\|\mu_{t+1}-\mu_t\dualnorm^2 \ge -\frac{\eta}{2\sigma}  \|z_t\primalnorm^2 + \frac{1}{\eta} V_h(\mu, \mu_{t+1})\ .
\end{align*}
Notice that $\|z_t\primalnorm\le G_2$, and we thus have because $a_t \ge 0$
\begin{align}\label{eq:second}
    S_1 = \sum_{t=1}^T a_t z_t^\top(\mu_t-\mu) \le \frac{\eta G_2^2}{2\sigma} \sum_{t=1}^T a_t + \frac{1}{\eta} \underbrace{\sum_{t=1}^T a_t\pran{V(\mu,\mu_t)-V(\mu,\mu_{t+1})}}_{S_1'}\ .
\end{align}
Furthermore, we have
\begin{align}\label{eq:third}
\begin{split}
    S_1'=&a_1 \sum_{t=1}^T \pran{V(\mu,\mu_t)-V(\mu,\mu_{t+1})} + \sum_{t=1}^T (a_t-a_1) \pran{V(\mu,\mu_t)-V(\mu,\mu_{t+1})} \\
    =& a_1 (V(\mu,\mu_1)-V(\mu,\mu_{T+1}))+ \sum_{t=1}^T\sum_{s=2}^t (a_s-a_{s-1}) \pran{V(\mu,\mu_t)-V(\mu,\mu_{t+1})} \\
    =&a_1 (V(\mu,\mu_1)-V(\mu,\mu_{T+1}))+ \sum_{s=2}^T\sum_{t=s}^T (a_s-a_{s-1}) \pran{V(\mu,\mu_t)-V(\mu,\mu_{t+1})} \\
    =&a_1 (V(\mu,\mu_1)-V(\mu,\mu_{T+1}))+ \sum_{s=2}^T (a_s-a_{s-1}) \pran{V(\mu,\mu_s)-V(\mu,\mu_{T+1})} \\
    =&a_1 (V(\mu,\mu_1)-V(\mu,\mu_{T+1}))- (a_T-a_1) V(\mu,\mu_{T+1}) + \sum_{s=2}^T (a_s-a_{s-1}) V(\mu,\mu_s) \\
    \le& a_1 V(\mu,\mu_1)- a_T V(\mu, \mu_{T+1}) + \max_{1\le s\le T} V(\mu, \mu_s) \sum_{s=2}^T (a_s - a_{s-1})^+\,,
\end{split}
\end{align}
where the second equation follows from $a_t - a_1 = \sum_{s=2}^t (a_s - a_{s-1})$, the third equation follows from exchanging the order of summation, the fourth from telescoping, and the inequality because Bregman divergences are non-negative. The result follows from combining \eqref{eq:first}, \eqref{eq:second}, and \eqref{eq:third}.
\end{proof}

\subsection{Analysis of PID Controllers}\label{sec:analysis-PID}

In order to utilize Theorem \ref{thm:adversarial}, we need to compute the value of $q_t$ from $\lambda_t$. {The next lemma provides a principal way for such computation using Z-transforms. Before stating the lemma, we provide some useful definitions. Let $z[a]=\sum_{t=0}^{\infty} a_t z^t$ and $z[b]=\sum_{t=0}^{\infty} b_t z^t$ be the Z-transforms of the sequences $(a_t)_{t\ge0}$ and  $(b_t)_{t\ge0}$, respectively. We interpret $z[a]$ and $z[b]$ as formal power series, which allows us to manipulate infinite series without worrying about convergence~\citep{niven1969formal}.  
We say that $z[a]\newequal z[b]$ if the coefficients of $z^t$ for $t \ge 0$ match, i.e., $a_t = b_t$ for all $t\ge0$. Moreover, we say $z[a] \newequalT b$ if the coefficients of $1, z, z^2,\ldots,z^{T-1}$ in these two formal power series are the same, i.e., $a_0=b_0, \ldots, a_{T-1}=b_{T-1}$.}

\begin{lem}\label{lemma:transform}
    Suppose $\lambda_0\not=0$. Denote $z[\lambda]=\sum_{t=0}^{T-1} \lambda_t z^t$ and $z[q]=\sum_{t=0}^{T-1} q_t z^t$ as two polynomial in $z$. Then, we have
    $$
    z[q]\newequalT\frac{1}{z[\lambda]}\,,
    $$
\end{lem}
\begin{proof}
     Denote by $D$ the subdiagonal shift matrix, i.e., $D_{i,j} = 1$ if $j = i-1$ and zero otherwise. 
Then, we have $R=\sum_{t=0}^{T-1} \lambda_t D^t$, and $R^{-1}=\sum_{t=0}^{T-1} q_t D^t$. By noticing $RR^{-1}=I$, we have
\begin{align}\label{eq:identity}
    I=\pran{\sum_{t=0}^{T-1} \lambda_t D^t}\pran{\sum_{t=0}^{T-1} q_t D^t}
    = \sum_{k=0}^{2T - 2} D^k \sum_{i=0}^k \lambda_i q_{k-i}
    = \sum_{k=0}^{T-1} D^k \sum_{i=0}^k \lambda_i q_{k-i}\,,
\end{align}
where the last equation follows because $D^T=0$. Therefore, we know that $q_0=\frac{1}{\lambda_0}$, and the coefficient of $D^k$ in the RHS of \eqref{eq:identity} must be 0 for $k=1,2,...,T-1$. This shows that $z[q] z[\lambda] \newequal \sum_{k=0}^{2T - 2} z^k \sum_{i=0}^k \lambda_i q_{k-i}  \newequalT \sum_{k=0}^{T-1} z^k \sum_{i=0}^k \lambda_i q_{k-i} \newequal 1$, which finishes the proof by noticing $\lambda_0\not=0$.
\end{proof}

The next four corollaries presents the regret bound for four algorithms: P controller (mirror descent), PD controller (mirror descent with optimism), PI controller (mirror descent with momentum) and PID controller (mirror descent with optimism and momentum).

\begin{cor}\label{cor:P}
    \textbf{P controller (mirror descent).}  We consider P controller where $\lambda_0=1$, $\lambda_t=0$ for $t=1,...,T-1$, thus $q_0=1$, $q_t=0$ for $t=1,...,T-1$, and $a_1=a_2=...=a_T=1$. Furthermore, we have $G_2=G_1$. It follows from Theorem \ref{thm:adversarial} that
$$
\Regret{\mu}\le \frac{\sqrt{2} \eta T G_1^2}{2\sigma} + \frac{1}{\eta} V(\mu, \mu_1)\ .
$$
\end{cor}

{
In the case of a P controller, we have that $R = \operatorname{diag}(\mathbf 1) = Q$. Therefore, $q_0 = 1$ and $q_t = 0$ for $t \ge 1$. Moreover, $a_t=q_0+\cdots+q_{T-t} = q_0 = 1$. The result follows from Theorem~\ref{thm:adversarial}. Our result matches, up to constant factors, the traditional regret bound of online mirror descent. Moreover, setting a step-size $\eta \sim T^{-1/2}$ would lead to regret $\Regret{\mu} = O(T^{1/2})$.} 

\begin{cor}\label{cor:PD}
    \textbf{PD controller (mirror descent with optimism).} Consider a PD controller, where 
\begin{equation*}
    \lambda_0=1, \lambda_1=-\alpha_D>-1, \lambda_2=...=\lambda_{T-1}=0\ .
\end{equation*}
Then, we have
$$
\Regret{\mu}\le \frac{\eta T G_2^2}{2\sigma} \frac{1}{1-\alpha_D} + \frac{1}{\eta(1-\alpha_D)} V(\mu, \mu_1) + \sqrt{2} G_1G_2 \frac{T \eta}{\sigma (1-\alpha_D)} \ .
$$
\end{cor}

{In the case of a PD controller, we can use Lemma~\ref{lemma:transform} to characterize inverse matrix $Q$. Here, we have that $z[\lambda] = \lambda_0 + \lambda_1 z$. Therefore, $z[q] \newequalT 1/z[\lambda] \newequal 1/(1-\alpha_D z)  \newequal \sum_{i=0}^{T-1} \alpha_D^i z^i $, where the last equation follows from the geometric series for formal power series. Therefore, we have $q_0=1$ and $q_t=\alpha_D^t>0$ for $t \ge0$. Furthermore, we have $a_t=q_0+\cdots+q_{T-t}$ is positive and monotonically decreasing in $t$. Furthermore, we have $a_t\le a_1\le \sum_{s=1}^\infty \alpha_D^s = 1/(1-\alpha_D)$, and
$$
\sum_{j=1}^T \sum_{k=1}^j k |a_j \lambda_k|=\sum_{j=1}^T |a_j \lambda_1| \le T |\lambda_1| a_1 \le \frac{T\alpha_D}{1-\alpha_D} \ .
$$
Thus, the corollary follows from Theorem \ref{thm:adversarial}. Whenever $\alpha_D < 1$, we obtain that a step-size $\eta \sim T^{-1/2}$ would lead to regret $\Regret{\mu} = O(T^{1/2})$.}


\begin{cor}\label{cor:PI} \textbf{PI controller (mirror descent with momentum)}
    We consider PI controller where \begin{align*}
    \lambda_0=1-\alpha_{I}+\alpha_I(1-\beta),
    \lambda_{i}=\alpha_I (1-\beta)\beta^i \text{ for } i=1,...,T-1\ ,
\end{align*}
with $0<\alpha_I, \beta<1$. Then we have
$$
\Regret{\mu}\le \frac{\eta G_2^2 T }{2\sigma \lambda_0}  + \frac{1}{\eta \lambda_0} V(\mu,\mu_1) +\frac{\beta-s}{\eta (1-s)} \max_{1\le s\le T} V(\mu, \mu_s) + \sqrt{2} G_1 G_2 \frac{\eta}{\sigma} \frac{\alpha_I T \beta}{\lambda_0(1-\beta)}\,\ ,
$$
where $s=\frac{(1-\alpha_I)\beta}{\lambda_0}$.
\end{cor}

\begin{proof}

Notice that
\begin{align*}
    z[\lambda]&\newequal\lambda_0+\sum_{i=1}^{T-1} \alpha_I (1-\beta)\beta^i z^i \newequal\lambda_0+\alpha_I(1-\beta) \beta z \pran{\frac{1-(\beta z)^{T-1}}{1-\beta z}} \newequalT \lambda_0+ \frac{\alpha_I(1-\beta) \beta z}{1-\beta z}\ ,
\end{align*}
and using Lemma~\ref{lemma:transform}
\begin{align*}
    z[q]&\newequalT\frac{1}{z[\lambda]}\newequalT\frac{1-\beta z}{\lambda_0 - \pran{\lambda_0-\alpha_I (1-\beta)} \beta z}\newequal\frac{1-\beta z}{\lambda_0 - \pran{1-\alpha_I} \beta z}\\
    &\newequal\frac{1-\beta z}{\lambda_0} \pran{1+\sum_{i=1}^\infty s^i z^i}\newequal\frac{1}{\lambda_0} + \sum_{i=1}^\infty (s^i-\beta s^{i-1})z^i\\
    &\newequal \frac{1}{\lambda_0} - \beta \frac{\alpha_I (1-\beta)}{\lambda_0} \sum_{i=1}^\infty s^{i-1} z^i
    \newequal \frac{1}{\lambda_0} - (\beta-s) \sum_{i=1}^\infty s^{i-1} z^i \ ,
\end{align*}
where $\omega=\frac{(1-\alpha_I)\beta}{\lambda_0}$. Therefore, we have 
\begin{align*}
    q_0=\frac{1}{\lambda_0}\ ,
    q_t=-(\beta-\omega) \omega^{t-1} .
\end{align*}
Thus, $a_t$ is non-negative and monotonically increasing in $t$, and 
\begin{align*}
    \frac{1}{\eta}\pran{\max_{1\le s\le T} V(\mu, \mu_s) \sum_{s=2}^T (a_s - a_{s-1})^+} &= \frac{1}{\eta} \max_{1\le s\le T} V(\mu, \mu_s) (a_T -a_1)  = \frac{1}{\eta}  \max_{1\le s\le T} V(\mu, \mu_s) \sum_{t=1}^{T-1}(\beta-\omega) \omega^{t-1} \\
    &\le \frac{\beta-\omega}{\eta (1-\omega)} \max_{1\le s\le T} V(\mu, \mu_s)\ .
\end{align*}
Furthermore, we have $\|z_t\primalnorm$
\begin{align*}
    \sum_{j=1}^T \sum_{k=1}^j k |a_j \lambda_k| \le a_T \sum_{j=1}^T \sum_{k=1}^j k  \lambda_k = q_0 \alpha_I (1-\beta) \sum_{j=1}^T \sum_{k=1}^j k   \beta^k \le q_0 \alpha_I (1-\beta) T \frac{\beta}{(1-\beta)^2}=\frac{\alpha_I T \beta}{\lambda_0(1-\beta)}.
\end{align*}

Plugging into \eqref{eq:RFTL} results in

$$
\Regret{\mu}\le \frac{\eta G_2^2 T }{2\sigma \lambda_0}  + \frac{1}{\eta \lambda_0} V(\mu,\mu_1) +\frac{\beta-\omega}{\eta (1-\omega} \max_{1\le s\le T} V(\mu, \mu_s) + \sqrt{2} G_1 G_2 \frac{\eta}{\sigma} \frac{\alpha_I T \beta}{\lambda_0(1-\beta)}\,\ .\qedhere
$$
\end{proof}

Next, we study a PID controller. The behavior of a PID controller is governed by the roots of the following quadratic characteristic function in $z$, which is the denominator of transfer function of $q$:
\begin{equation}\label{eq:quadratic}
    \alpha_D \beta z^2 -\pran{\alpha_D + \beta(1-\alpha_I)} z + 1-\alpha_I \beta\ .
\end{equation}
The next assumption guarantees that the quadratic function \eqref{eq:quadratic} has two real roots, which corresponds to the convolutional linear filter being an overdamped oscillator. We denote $z_+\ge z_-$ as the two real roots of the quadratic function in \eqref{eq:quadratic}. Under this assumption, the coefficients $q_t$ and $a_t$ do not oscillate and $a_t \ge 0$. The condition in Assumption~\ref{ass:real-roots} trivially holds when $\alpha_D$ or $\alpha_I$ is zero. If this condition does not hold, $z_+$ and $z_-$ can be complex, and $a_t$ could oscillate, and some terms can be even negative.

\begin{ass}[Real roots]\label{ass:real-roots}
    The parameters satisfy $(\alpha_D+\beta(1-\alpha_I))^2\ge 4\alpha_D \beta (1-\alpha_I\beta)$.
\end{ass}

{ We next provide our regret bound for a PID controller when the characteristic function has real roots. As before, we obtain that a step-size $\eta \sim T^{-1/2}$ leads to regret $\Regret{\mu} = O(T^{1/2})$. While our regret bound is weaker than that of the pure P controller in Corollary~\ref{cor:P}, our numerical results in the next section suggest that PID controllers can lead to higher performance than a pure P controller.}

\begin{cor}\label{cor:PID}
    \textbf{PID controller (mirror descent with optimism and momentum).} 
We consider PID controller with $\alpha_D, \alpha_I >0$, $\alpha_D + \alpha_I \le 1$ and $0<\beta<1$. Thus, we have 
\begin{align*}
    \lambda_0&=1-\alpha_{I}+\alpha_I(1-\beta)\\
    \lambda_1&=-\alpha_D+\alpha_I(1-\beta) \beta \\
    \lambda_{i}&=\alpha_I (1-\beta)\beta^i \text{ for } i=2,...,T-1\ .
\end{align*}
Suppose Assumption~\ref{ass:real-roots} holds. Denote $M:=\frac{1}{c}+\frac{z_+}{ c \pran{1-\frac{az_+}{c}}^2}+\frac{\beta z_+}{ c \pran{1-\frac{az_+}{c}}^2}$. Then, it holds that
$$
\Regret{\mu}\le \frac{\eta G_2^2 T M }{2\sigma}  + \frac{M}{\eta} V(\mu,\mu_1) +  \frac{M}{\eta}\max_{1\le s\le T} V(\mu, \mu_s) + \sqrt{2} G_1 G_2 \frac{\eta M}{\sigma} \pran{T \alpha_D + \alpha_I  T \frac{\beta}{1-\beta}}\,\ .
$$
\end{cor}

\begin{proof}


Notice that
\begin{align*}
    z[\lambda]&\newequal 1-\alpha_I -\alpha_D z + \sum_{i=0}^{T-1} \alpha_I (1-\beta)\beta^i z^i \newequal 1-\alpha_I -\alpha_D z + \frac{\alpha_I (1-\beta)}{1-\beta z} \ ,
\end{align*}
and by Lemma~\ref{lemma:transform}
\begin{align*}
z[q] \newequalT 1/z[\lambda]   \newequal\frac{1-\beta z}{\alpha_D \beta z^2 -\pran{\alpha_D + \beta(1-\alpha_I)} z + 1-\alpha_I \beta} \ . 
\end{align*}
Denote $a= \alpha_D \beta$, $b= \alpha_D + \beta(1-\alpha_I)$ and $c=1-\alpha_I \beta$, then
$$z[q]\newequalT\frac{1-\beta z}{az^2-bz+c}\ .$$
Recall $z_+\ge z_-$ are the two roots of the quadratic function $az^2-bz+c$ in $z$. That is, $z_+=\frac{b+\sqrt{b^2-4ac}}{2a}$ and $z_-=\frac{b-\sqrt{b^2-4ac}}{2a}$. It is easy to check that $z_+, z_-\ge 0$. Suppose $z_+$ and $z_-$ are distinct, then,
{
\begin{align*}
    z[q]&\newequalT\frac{1-\beta z}{a} \frac{1}{z-z_+} \frac{1}{z-z_-}\newequal\frac{1-\beta z}{a(z_+-z_-)} \pran{\frac{1}{z-z_+}-\frac{1}{z-z_-}}\\
    &\newequal\frac{1-\beta z}{a(z_+-z_-)} \pran{-\frac{1}{z_+}\sum_{i=0}^\infty (z/z^+)^i + \frac{1}{z_-}\sum_{i=0}^\infty (z/z^-)^i}\\
    &\newequal\frac{1-\beta z}{a(z_+-z_-)} \pran{\sum_{i=0}^\infty z^i \pran{-\frac{1}{z_+^{i+1}}+ \frac{1}{z_-^{i+1}}}}\newequal\frac{1-\beta z}{a(z_+-z_-)} \pran{\sum_{i=0}^\infty z^i \pran{\frac{a}{c}}^{i+1} \pran{z_+^{i+1}-z_-^{i+1}}}\\
    &\newequal\frac{1}{a(z_+-z_-)} \pran{ \frac{a}{c}\pran{z_+-z_-}+ \sum_{i=1}^\infty z^i \pran{\pran{\frac{a}{c}}^{i+1} \pran{z_+^{i+1}-z_-^{i+1}} - \beta \pran{\frac{a}{c}}^{i} \pran{z_+^{i}-z_-^{i}} }}
\end{align*}
}%
where the fifth equality uses $z_+ z_-=\frac{c}{a}$. Therefore, we have 
\begin{align*}
    q_0&=\frac{1}{c}\ ,
    q_i=\frac{1}{a(z_+-z_-)} \pran{\pran{\frac{a}{c}}^{i+1} \pran{z_+^{i+1}-z_-^{i+1}} - \beta \pran{\frac{a}{c}}^{i} \pran{z_+^{i}-z_-^{i}} } \text{ for } i\ge 1 \ .
\end{align*}
Thus, because $z_+ > z_-$ we obtain
{\small
\begin{align*}
    a_t&=q_0+\cdots+q_{T-t}\\
    &=(1-\beta) \frac{1}{a(z_+-z_-)} \pran{\sum_{i=0}^{T-t}\pran{\frac{a}{c}}^{i} \pran{z_+^{i}-z_-^{i}}}+\frac{1}{a(z_+-z_-)} \pran{\pran{\frac{a}{c}}^{T-t+1} \pran{z_+^{T-t+1}-z_-^{T-t+1}}  } \ge 0\ .
\end{align*}
}
Notice that $\frac{z_+^i-z_-^i}{z_+-z_-}=z_+^{i-1}+z_+^{i-2}z_-+\cdots + z_-^{i-1}\le i z_+^{i-1}$, thus we have
\begin{align*}
    |q_i|\le \frac{1}{a} \pran{\pran{\frac{a}{c}}^{i+1} (i+1) z_+^{i+1} + \beta \pran{\frac{a}{c}}^{i} i z_+^{i}}\ .
\end{align*}
Furthermore, using the formula for $z_+$ we obtain $$\frac{az_+}{c} = \frac{b+\sqrt{b^2-4ac}}{2c}< \frac{b+\sqrt{b^2-4bc+4c^2}}{2c}=\frac{b+|2c-b|}{2c}= 1\ ,$$
where the inequality uses $a+c>b$ and $c>0$, and the last equality uses $$2c-b=2-\alpha_D-\beta-\beta \alpha_I\ge 2-\alpha_D-\beta-\alpha_I\ge 0\ .$$
Therefore, we have $\frac{az_+}{c}<1$, and thus
\begin{align*}
    \sum_{i=0}^{T-1} |q_i|&\le \frac{1}{c} + \frac{1}{a} \pran{\sum_{i=1}^{T-1} (i+1) \pran{\frac{az_+}{c }}^{i+1}}+ \frac{\beta}{a} \pran{\sum_{i=1}^{T-1} i \pran{\frac{az_+}{c }}^{i}}\le \frac{1}{c}+\frac{z_+}{ c \pran{1-\frac{az_+}{c}}^2}+\frac{\beta z_+}{ c \pran{1-\frac{az_+}{c}}^2}=M\ .
\end{align*}
Suppose the two roots are the same, i.e., $z_+=z_-=\frac{b}{2a}>0$. We have 

\begin{align*}
    z[q]&\newequalT\frac{1-\beta z}{a(z-z_+)^2}\newequal\frac{1-\beta z}{az_+^2}\sum_{i=0}^{\infty} \frac{(i+1)z^i}{z_+^i}\newequal \frac{1}{az_+^2} + \sum_{i=1}^\infty \frac{1}{a z_+^2}\pran{\frac{i+1}{z_+^i}-\frac{\beta i}{z_+^{i-1}}}z^i\ .
\end{align*}
Thus 
\begin{align*}
    q_0=\frac{1}{az_+^2} \ ,q_i=\frac{1}{a z_+^2}\pran{\frac{i+1}{z_+^i}-\frac{\beta i}{z_+^{i-1}}} \text{ for } i\ge 1\ .
\end{align*}
Since $z_+=\frac{b}{2a}=\sqrt{\frac{c}{a}}>1$, we have
\begin{align*}
    \sum_{i=0}^{T-1} |q_i|&\le \frac{1}{c} + \frac{1}{c} \pran{\sum_{i=1}^{T-1} \pran{\frac{i+1}{z_+^i}}+ {\beta} \pran{\sum_{i=1}^{T-1} \frac{i}{z_+^{i-1}}}}\le \frac{1}{c}\pran{1+(1+\beta)/\pran{1-\frac{1}{z_+}}^2} =M\ .
\end{align*}
Therefore, in both cases, we have 
\begin{equation*}
    \sum_{i=0}^{T-1} |q_i|\le M\ .
\end{equation*}
Thus, it holds that
\begin{align}\label{eq:pid-eq1}
    \frac{1}{\eta}\pran{\max_{1\le s\le T} V(\mu, \mu_s) \sum_{s=2}^T (a_s - a_{s-1})^+} &\le \frac{1}{\eta}\pran{\max_{1\le s\le T} V(\mu, \mu_s) \sum_{s=2}^T |q_{T-s+1}|}\le \frac{M}{\eta}\max_{1\le s\le T} V(\mu, \mu_s)\ .
\end{align}
Furthermore, we always have $a_j=q_0+\cdots+q_{T-t}\le M$, thus
\begin{align}\label{eq:pid-eq2}
    G_1 G_2 \frac{\eta}{\sigma} \sum_{j=1}^T \sum_{k=1}^j k |a_j \lambda_k|&\le G_1 G_2 \frac{\eta M}{\sigma} \sum_{j=1}^T \sum_{k=1}^j k | \lambda_k|\le G_1 G_2 \frac{\eta M}{\sigma} \pran{T \alpha_D + \alpha_I (1-\beta) \sum_{j=1}^T \sum_{k=1}^j k   \beta^k} \\
    &\le G_1 G_2 \frac{\eta M}{\sigma} \pran{T \alpha_D + \alpha_I (1-\beta) T \frac{\beta}{(1-\beta)^2}}=G_1 G_2 \frac{\eta M}{\sigma} \pran{T \alpha_D + \alpha_I  T \frac{\beta}{1-\beta}}\ .\nonumber
\end{align}
Plugging \eqref{eq:pid-eq1} and \eqref{eq:pid-eq2} into \eqref{eq:RFTL}, we arrive at
$$
\Regret{\mu}\le \frac{\eta G_2^2 T M }{2\sigma}  + \frac{M}{\eta} V(\mu,\mu_1) +  \frac{M}{\eta}\max_{1\le s\le T} V(\mu, \mu_s) + \sqrt{2} G_1 G_2 \frac{\eta M}{\sigma} \pran{T \alpha_D + \alpha_I  T \frac{\beta}{1-\beta}}\,\ ,
$$
which finishes the proof.
\end{proof}

\section{Back to Online Allocation Problems}
\label{sec:back_to_online}

In this section, we complete the picture by providing our final result that establishes the first performance guarantees for the dual-based PID controller (Algorithm \ref{al:PID}) on online resource allocation problems. As mentioned earlier in the introduction, the key ideas we use for this final step are those from~\cite{balseiro2020best}. For the online allocation problem \eqref{eq:OPT}, we consider a data-driven regime in which requests are drawn independently from a probability distribution $\cP\in \Delta(\cS)$ that is unknown to the decision maker, where $\Delta(\cS)$ is the space of all probability distributions over support set $\cS$. An online algorithm $A$ makes, at time $t$, a real-time decision $x_t$ based on the current request $(f_t, b_t, \cX_t)$ and the history. We define the reward of an algorithm for input $\vgamma$ as $R(A | \vgamma) = \sum_{t=1}^T f_t(x_t)$,
where $x_t$ is the action taken by algorithm $A$ at time $t$. Moreover, the algorithm $A$ must satisfy constraints $\sum_{t=1}^{T} b_{t}(x_{t}) \le B$ and $x_{t}\in \cX$ for every $t\le T$. Let $\rho_j = B_j/T$ be the average availability of resource $j$. We denote by $\lbrho = \min_{j \in [m]} \rho_j$ the lowest resource parameter and $\ubrho= \max_{j \in [m]} \rho_j$ the largest resource parameter. We measure the regret of an algorithm as the worst-case difference over distributions in $\Delta(\cS)$, between the expected performance of the benchmark and the algorithm:
    \begin{align*}
    \Regret{A} = \sup_{\cP \in \Delta(\cS)}  \left\{ \EE_{\vgamma \sim \cP^T} \left[ \OPT(\vgamma) - R(A|\vgamma) \right] \right\}\,.
    \end{align*}
    

The next assumption imposes some regularity conditions on the requests. Throughout this section, we use the $\|\cdot\|_\infty$ norm as primal norm and $\|\cdot\|_1$ as dual norm.

\begin{ass}[Regularity conditions on the requests]
    \label{ass:p}
	There exists $\ubf\in\RR_{+}$ and $\ubb \in \RR_+$ such that for all requests $(f,b,\cX) \in \cS$ in the support, it holds that 
	\begin{enumerate}
 	    \item The feasible set satisfies $0\in \cX$.
 	    \item The reward functions satisfy $0 \le f(x)\le \ubf$ for every $x\in\cX$.
 	    \item The resource consumption functions satisfy $b(x)\ge 0$ and $\|b(x)\|_{\infty} \le \ubb$ for every $x\in\cX$.
 	    \item The optimization problems in \eqref{eq:primal_decision} admit an optimal solution.
 	\end{enumerate}
\end{ass}

The following assumption on the reference function $h$ is a refinement of Assumption~\ref{ass:h} and is needed to develop our theoretical guarantees.

\begin{ass}[Separability of the reference function $h$]\label{ass:h-sep} The reference function $h(\mu)$ is coordinate-wisely separable, i.e., $h(\mu)=\sum_{j=1}^m h_j(\mu_j)$ where $h_j:\RR_+ \rightarrow \RR $ is an univariate function, which is either differentiable or essentially smooth. Moreover, for every resource $j$ the function $h_j$ is $\sigma$-strongly convex over $[0, \mumax_j]$ with $\mumax_j:= \ubf / \rho_j + 4\eta(\ubb + \rho_j) / (\sigma (1-\beta))$ and $\rho_j = B_j/T$.
\end{ass}

The reference function used in Proposition~\ref{prop:PID-equi} to show the equivalence between Algorithm~\ref{al:PID} and online mirror descent is separable by assumption because $h_j(\mu_j) = \int_0^{\mu_j} \ell_j(s) ds$ and convex whenever the functions $\ell_j$ are increasing. For the strong convexity of $h_j(\mu_j)$, it is enough that $\ell_j'(s) \ge \sigma$. 

We next discuss the implications of different choices for the reference functions. One natural choice is the squared-Euclidean norm, which yields online gradient descent with momentum and, by Proposition~\ref{prop:PID-equi}, the additive PID controller in \eqref{eq:api}. Because the reference function $h(\mu) = \| \mu\|_2^2/2$ is strongly convex over $\mathbb R_+^m$, then Assumption~\ref{ass:h-sep} is satisfied and our results apply. Another natural choice is using the negative entropy $h(\mu) = \sum_{j=1}^m \mu_j \log(\mu_j)$, which by Proposition~\ref{prop:PID-equi} would lead to the multiplicative PID controller in \eqref{eq:mpi}. Unfortunately, this function is not strongly convex over $\mathbb R_+^m$ because the curvature of the negative entropy goes to zero for large values of $\mu$, i.e., it is asymptotically linear. To apply this algorithm, we need a more sophisticated analysis of the evolution of the iterates.

Below we present an analysis of the evolution of the iterates with the goal of proving that these remain bounded through the run of the algorithm. The proof of the below lemma can be found in Appendix \ref{app:missing}.

\begin{lem}\label{lemma:iterates}
Suppose that Assumption~\ref{ass:h-sep} holds and for resource $j$ the initial conditions satisfy $\mu_{1,j} \le \ubf / \rho_j$ and $z_{1,j} = 0$. Then we have that $\mu_{t,j} \le \mumax_j$ for all $t\ge 1$.
\end{lem}

The previous results have two important consequences for the analysis of dual-based PID controllers. Firstly, since the iterates $\mu_t$ stay in a bounded region, the third term in regret bound of CMD in Theorem \ref{thm:adversarial}, which depends on $V(\mu,\mu_t)$, can be upper bounded by a constant when $\mu$ is bounded. This is required to apply the machinery developed in \cite{balseiro2020best}. Secondly, it allows one to apply Algorithm~\ref{al:PID} to exponential weights (i.e., \eqref{eq:mpi}) because the reference function $h_j(\mu_j) = \mu_j \log(\mu_j)$ is $(\mumax_j)^{-1}-$strongly convex over the set $[0,\mumax_j]$. Therefore, we can restrict the algorithm to the box $\mathcal U:=\prod_{j=1}^m [0,\mumax_j]$ without loss of optimality. A similar analysis can be extended to other non-linear updates such as the one described in \eqref{eq:non-linear-update}, which do not necessarily induce strongly convex reference functions over the non-negative orthant. For example, in the case of budget pacing, the popular multiplicative pacing strategy given in \eqref{eq:alternative-pacing} leads to the reference function $h(\mu) = \int_0^\mu \ell(s) ds = q (1+\mu)\log(1+\mu)-\mu(q+1)$. This reference function is not strongly convex over the non-negative reals, but it is $q (1+\mumax)^{-1}-$strongly convex over the set $[0,\mumax]$.




Theorem \ref{thm:master} presents the regret bound of the dual-based PID controller for online allocation problems.


\begin{thm}\label{thm:master}
Suppose requests are drawn independently from a fixed unknown distribution, the reference function $h$ used in the update satisfies Assumption \ref{ass:h-sep}, and the parameters of the PID controller satisfy Assumption~\ref{ass:real-roots}, $\beta \in [0,1)$, $\alpha_D, \alpha_I \ge 0$, $\alpha_D<1$ and $\alpha_D+\alpha_I\le 1$. If the step-size satisfies $\eta = \Theta(T^{-1/2})$ and initial resources scale linearly with $T$, then it holds that
\[
\Regret{A} = O(T^{1/2})\ .
\]
\end{thm}

\begin{proof}
    It follows from Lemma \ref{lemma:iterates} that the dual variable $\mu_{t,j}\le \mumax_j$ stays in a bounded region for any $t\ge 1$. Thus, Assumption~\ref{ass:h-sep} implies that $h_j$ is $\sigma$-strongly convex. Next, we will invoke Theorem 1 in \cite{balseiro2020best} to link $\Regret{A}$ with $\Regret{\mu}$, and then utilize Theorem \ref{thm:adversarial} to bound  $\Regret{\mu}$. More specifically, Assumption \ref{ass:p} and the $\sigma$-strong-convexity of $h_j$ guarantee that the assumptions of Theorem 1 in \cite{balseiro2020best} hold.   
    It then follows from Theorem 1 in \cite{balseiro2020best} that 
    \begin{align}\label{eq:regret-connection}
        \Regret{A}\le \frac{\ubf \ubb}{\lbrho} + \max_{\mu \in \{0, (\bar f/\rho_1) e_1,\ldots, (\bar f/\rho_m) e_m\}}\Regret{\mu} \ ,
    \end{align}
    where $e_j\in\RR^m$ is the $j$-th unit vector. Notice that in the case when $\alpha_I=\alpha_D=0$, we recover a P~controller, and a bound on $\Regret{\mu}$ is given in Corollary \ref{cor:P}; when $\alpha_I=0$ and $\alpha_D>0$, we recover a PD~controller, and a bound on $\Regret{\mu}$ is given in Corollary \ref{cor:PD}; when $\alpha_I>0$ and $\alpha_D=0$, we recover a PI~controller, and a bound on $\Regret{\mu}$ is given in Corollary \ref{cor:PI}; when $\alpha_I, \alpha_D>0$, we recover a PID~controller, and a bound on $\Regret{\mu}$ is given in Corollary \ref{cor:PID}. In all these four cases, we have that $\max_{\mu \in \{0, (\bar f/\rho_1) e_1,\ldots, (\bar f/\rho_m) e_m\}}\Regret{\mu} \le O(\eta T)+ O(\frac{1}{\eta})= O(\sqrt{T})$ in the case when $\eta =\Theta(\sqrt T)$ because $\rho$ is fixed. This finishes the proof with \eqref{eq:regret-connection}.
\end{proof}

	
Theorem~\ref{thm:master} implies asymptotic optimality when the expected offline performance satisfies $\EE_{\vgamma \sim \cP^T} \left[ \OPT(\vgamma) \right] = \Theta(T)$ and initial resources scale linearly with $T$, i.e., they satisfy $B_j = \rho_j T$ for some fixed $\rho_j > 0$. The expected offline performance grows linearly, for example, if for every request $\gamma$ there exists some action $x(\gamma)$ that generates a positive expected reward, i.e., $\EE_{\gamma \sim \mathcal P}[ f(x(\gamma))] > 0$. Under these assumptions, the dual-based PID controller achieves an asymptotic competitive ratio of one, i.e., 
\[
    \lim_{T\rightarrow \infty} \inf_{\cP \in \Delta(\cS)}  \frac{ \EE_{\vgamma \sim \cP^T} \left[ R(A|\vgamma) \right]} {\EE_{\vgamma \sim \cP^T} \left[ \OPT(\vgamma) \right]} = 1 \,.
\]

{
\subsection{Numerical Experiments}

We evaluate the performance of our algorithm on a set of random online linear programming problems, which are motivated by the dynamic allocation of advertising campaigns. The action set is the unit simplex $\cX = \{ x \in \RR_+^d : \sum_{i=1}^d x_i \le 1\}$. The reward and consumption functions are both linear and given by $f_t(x) = r_t^\top x$ and $b_t(x) = c_t x$, respectively, where $r_t \in \RR_+^d$ is a reward vector and $c_t \in \RR_+^{m \times d}$ is a consumption matrix. As in Section~\ref{sec:back_to_online}, we want $\sum_{t=1}^T b_t(x_t) \leq B$ where $B \in R^m_+$ and $\rho_j = B_j /T$ is the average availability of resource $j$ over the horizon of length $T$. 

{
\paragraph{Data Generation Process.} There are two layers of randomness in our experiments: we first randomly draw parameters of the underlying instance, and then we sample multiple sample paths of reward vectors and resource consumption matrices from these distributions.

Each entry of the consumption matrix $c_t$ is drawn independently from a Bernoulli distribution with a row-dependent probability parameter $p_j$ for each resource $j=1,\ldots,m$. The reward vector is given by $r_t=\beta \cdot \left(\theta^\top c_t + \delta_t \mathbf 1\right)$, where $\beta \ge 0$ is a scale parameter, $\theta \in \RR^m$ is a vector that correlates consumption $c_t$ with reward $r_t$, and $\delta_t \in \RR$ is an i.i.d.~Gaussian noise with mean zero and standard deviation 0.1. Therefore, an instance is characterized by the tuple $I = (p,\beta, \theta, \rho)$ with $p \in [0,1]^m$, $\beta \in \RR_+$, $\theta \in \RR^m$, and $\rho \in \RR_+^m$.

Finally, the parameters of the instance $I$ are generated as follows. We set $p_j \sim \operatorname{Uniform}(0, 1)$ and $\rho_j \sim \operatorname{Uniform}(0,1)$ for each resource $j = 1,\ldots,m$. We generate $\theta$ from a standard multi-variate Gaussian distribution $N(0,\text{diag}(\mathbf 1))$ and then it is normalized to satisfy $\|\theta\|_2=1$. Finally, the scale parameter $\beta$ is drawn from a log-normal distribution with mean 0 and standard deviation 2. The variability in the scale parameters makes the allocation problems harder by forcing the optimal objective values and optimal dual variables to have different orders of magnitude across instances, and reduces the possibility of overfitting the PID controller's parameters to a particular instance.}

For our experiments, we chose the number of resources to $m=10$ and the dimension of the action set to $d=5$. We generated 20 different sets of parameters for the first layer and then simulated 10 trials for each set of parameters. We tried different lengths of horizon but reported results for $T=1,000$---the results for other lengths of horizons were similar and thus omitted.

\paragraph{PID Controller Parameters.} In our algorithm, we chose the squared-Euclidean norm, yielding the additive PID controller in \eqref{eq:api}. We tried different combinations for the parameters $\alpha_P$, $\alpha_I$, $\alpha_D$, and the momentum parameter $\beta$. We set the step-size to be $\eta = s T^{-1/2}$, where $s \in \mathbb R^+$ is a multiplier of the step-size. We try the following combinations of parameters: $s \in \{0.1, 1, 10, 100\}$, $\beta \in \{0, 0.9, 0.99, 0.999\}$, $\alpha_D \in \{0, 0.25, 0.5, 0.75\}$, $\alpha_I \in \{0, 0.25, 0.5, 0.75\}$, and $\alpha_P = 1 - \alpha_D - \alpha_I$. Combinations that lead to negative values of $\alpha_P$ are excluded, i.e., we only consider $\alpha_I + \alpha_D \le 1$. The values of the gain parameters $K_P,K_I,K_D$ can be recovered using Proposition~\ref{prop:PID-equi}. 

For each combination of parameters, we run $20 \times 10 = 200$ sample paths $\vgamma$ with $T=1,000$ request each. For each sample path $\gamma$, we calculate the reward of the algorithm $R(A|\vgamma)$, the offline optimum $\OPT(\vgamma)$ by solving a linear program, and the pathwise competitive ratio $R(A|\vgamma)/\OPT(\vgamma)$.  Average competitive are reported in Table~\ref{tab:results} and selected competitive ratios for PI, PD, and PID controllers are reported in Figure~\ref{fig:experiments}. The black error bars provide $95\%$ confidence intervals.

\begin{table}[p]
    \centering
        \begin{tabular}{ccc|cccc}%
        &&& \multicolumn{4}{c}{Competitive ratios}\\
        &&& \multicolumn{4}{c}{for different step-size multipliers}\\\hline
        $\alpha_I$ & $\alpha_D$ & $\beta$ & $s=0.1$ & $s=1$ & $s=10$ & $s=100$\\\hline
        \begin{filecontents*}{resultsvar2-10x20x1000-pivot-ssc.csv}
alpha_I,alpha_D,momentum_beta,0.1,1.0,10.0,100.0
0.0,0.0,0.0,0.8132702310087532,0.9186501090415128,0.8568631243292698,0.6918874107970144
0.0,0.25,0.0,0.7880094908799345,0.9055396391147531,0.8520960080073853,0.681347814791708
0.0,0.5,0.0,0.7531282866306414,0.8820928259240296,0.84880522116332,0.6901013816099437
0.0,0.75,0.0,0.7023352729497121,0.8441760279778967,0.8419039593924741,0.7117935133783323
0.0,1.0,0.0,0.5451985687773969,0.5543751152884604,0.5286860860613772,0.4644772956426091
0.25,0.0,0.9,0.8127148114073915,0.9237255087435041,0.8873341314012839,0.7425357597411303
0.25,0.0,0.99,0.8083861106986041,0.9227450380357378,0.8913453780541107,0.7483423418685909
0.25,0.0,0.999,0.7946305575451181,0.9163799496397822,0.880467280844684,0.73258268408337
0.25,0.25,0.9,0.7871212692774138,0.9123784989083785,0.8728728268220088,0.7200011161494004
0.25,0.25,0.99,0.7811482932401463,0.9098766427559957,0.8760570998171072,0.7269509435241663
0.25,0.25,0.999,0.7623993125410053,0.89995845034862,0.8694996644749221,0.7112050600588928
0.25,0.5,0.9,0.7521371679929928,0.8913515089618004,0.8656905328033419,0.7262845748680163
0.25,0.5,0.99,0.7432565321524801,0.8854682781772891,0.8680825443499046,0.7357952516247772
0.25,0.5,0.999,0.7157673676890371,0.868811962240995,0.860896698921396,0.7263826917852496
0.25,0.75,0.9,0.7005421201030998,0.8515598543265981,0.8514964011307447,0.7493818510249193
0.25,0.75,0.99,0.6864475535018595,0.8323909586029136,0.8369392818173295,0.7501888284742364
0.25,0.75,0.999,0.6222166446667322,0.7181267379358447,0.7531435344969902,0.7137748227561245
0.5,0.0,0.9,0.8121316756072665,0.9262095735394859,0.9162643357318241,0.7942664376061177
0.5,0.0,0.99,0.8026809237097176,0.9231716117298795,0.9250188884880718,0.8090859158348241
0.5,0.0,0.999,0.7707304936848339,0.9095538460813911,0.912299536893897,0.7830697957753142
0.5,0.25,0.9,0.786360580653159,0.9155802637277783,0.9005268423121264,0.7800872915445966
0.5,0.25,0.99,0.7730369772140062,0.9086489452011425,0.9062994178049473,0.7962990930215587
0.5,0.25,0.999,0.7272519353503175,0.8857464765179437,0.8868730148606703,0.7645970893339445
0.5,0.5,0.9,0.7509945132063842,0.8957458991535574,0.8776831017230319,0.7574520169771514
0.5,0.5,0.99,0.7304201458046675,0.8736921785888078,0.8633169460717209,0.7658734686570772
0.5,0.5,0.999,0.6431528618793475,0.7650268228027379,0.7758931239743238,0.7260783215943745
0.75,0.0,0.9,0.8113545515226616,0.926825220722131,0.935303960908847,0.8444019194856078
0.75,0.0,0.99,0.7954100266342696,0.9200576467809759,0.9492183922886284,0.8760383321398842
0.75,0.0,0.999,0.7377744859357778,0.8941971124401542,0.9424186147694406,0.8580105506569639
0.75,0.25,0.9,0.7855557720818135,0.9157371579929791,0.9277173197552772,0.8301076476011405
0.75,0.25,0.99,0.7631534525503477,0.8934091190474223,0.9169721796333998,0.845625387115862
0.75,0.25,0.999,0.6554499280433964,0.7757788615301733,0.816750643752743,0.7885843066813503
1.0,0.0,0.9,0.8108239941444323,0.924857587699905,0.9326743866235155,0.868679046373014
1.0,0.0,0.99,0.7863273524608637,0.9022827330566362,0.926284378132269,0.8979823494501961
1.0,0.0,0.999,0.6663340746481095,0.7821658219140609,0.8441908494789562,0.8774576536745468
\end{filecontents*}
\csvreader[
        head to column names
        ]{resultsvar2-10x20x1000-pivot-ssc.csv}{momentum_beta=\mb,alpha_D=\aD,alpha_I=\aI,0.1=\sa,1.0=\sb,10.0=\sc,100.0=\sd}{%
        \aI & \aD & \mb & \ApplyGradient{\sa}{\num[round-mode=places,round-precision=3]{\sa}} & \ApplyGradient{\sb}{\num[round-mode=places,round-precision=3]{\sb}} & \ApplyGradient{\sc}{\num[round-mode=places,round-precision=3]{\sc}} & \ApplyGradient{\sd}{\num[round-mode=places,round-precision=3]{\sd}}\\
        }%
        \end{tabular}
    \caption{Average competitive ratios for different combinations of parameters of a PID controller.}
    \label{tab:results}
\end{table}



 \begin{figure}
     \centering
     \subcaptionbox{PI controller ($\alpha_D = 0$ and $\alpha_I = 0.5$) for different values of $\beta$ and step-sizes multipliers $s$.\label{fig:experiments-PI}}{
     \includegraphics[width=0.45\textwidth]{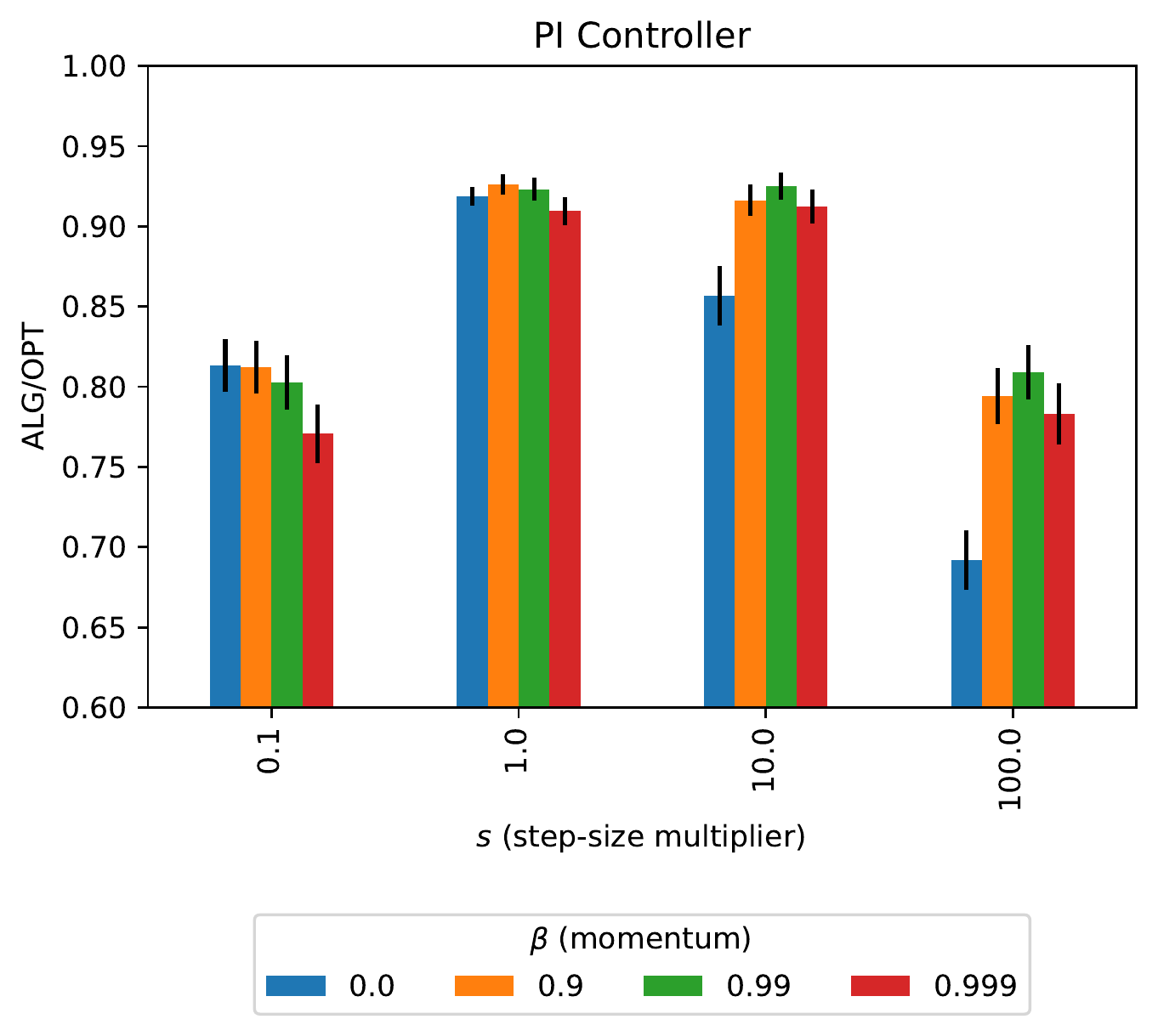}}\hfill
     \subcaptionbox{PD controller ($\alpha_I = 0$) for different values of $\alpha_D$ and step-sizes multipliers $s$.\label{fig:experiments-PD}}{
     \includegraphics[width=0.45\textwidth]{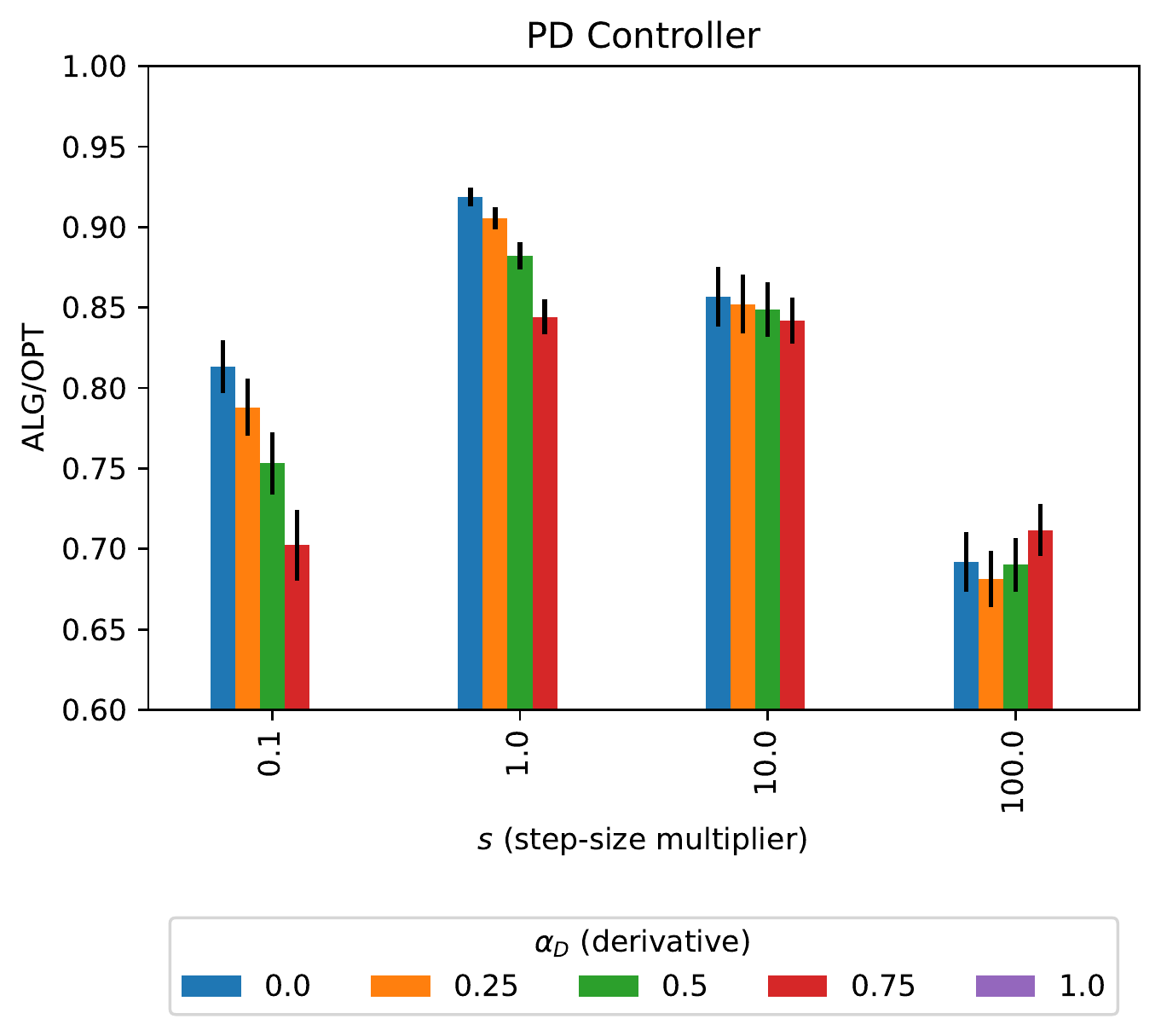}}  \hfill
     \subcaptionbox{PID controller for different values of $\alpha_D$ and $\alpha_I$ ($\alpha_I + \alpha_D \le 1$) with $s=10$ and $\beta=0.99$.\label{fig:experiments-PID}}{
     \includegraphics[width=0.45\textwidth]{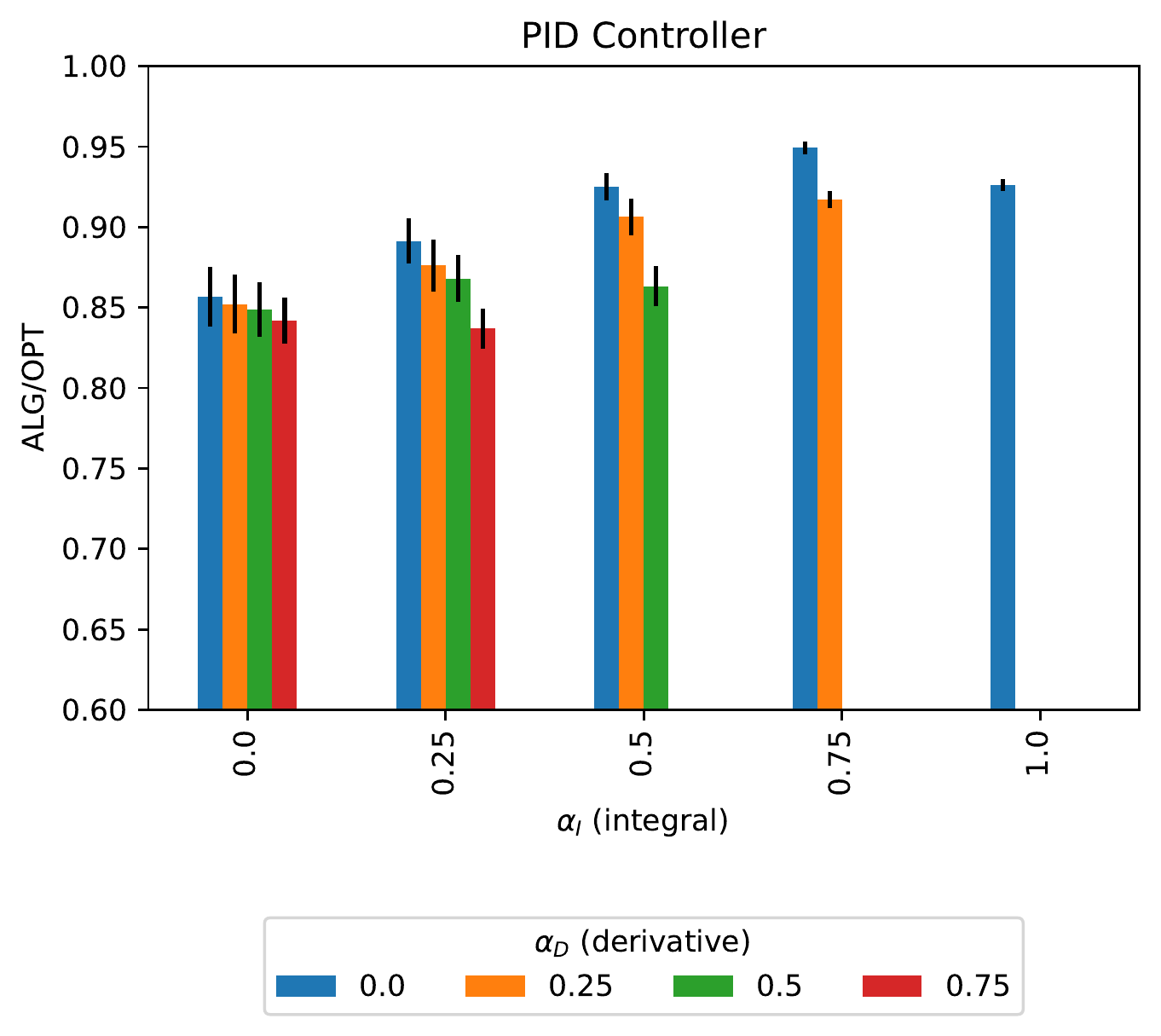}}   
     \caption{Average competitive ratios for Algorithm~\ref{al:PID}.}
     \label{fig:experiments}
 \end{figure}

We next summarize the main findings from our experiments:

(1) \emph{Summary.} The highest performance is achieved by a PID controller with step-size multiplier $s=10$ and $\alpha_I = 0.75$, which achieves a relatively high competitive ratio of $0.949$. Introducing a derivative term does not lead to higher performance. In contrast, a pure P controller obtains a competitive ratio of $0.919$ performance when properly tuned (with a step-size multiplier of $s=10$). The pure P controller performs poorly if the step-size is not properly tuned. The best PID controller performs 3\% better than a pure P controller and, as we discuss below, is more robust to step-size misspecification.

    
(2) \emph{PI controller (Figure~\ref{fig:experiments-PI}).} Around the optimal step-size multiplier of $s=10$, increasing the momentum, i.e., increasing the parameter $\beta$ in the integral term, yields modest improvements in performance. When the step-size is small, introducing momentum has little impact on performance. Large values of $\beta$, however, can decrease performance. This is expected as momentum can add too much friction to the algorithm, and the dual variables might take too long to converge to optimal values. Momentum has non-trivial influence on performance when the step-size is large. After a transient period, the dual variables would perform a random walk around the orbit of the dual variables. When the step-size is large, the time to reach this orbit is short, but the steps of the random walk would be of order $\eta \| z_t \|_2$. Introducing momentum averages the values of the gradient, which reduces the size of the orbit to order $\eta (1-\beta)^{1/2}$. This leads to increased expected performance when the step-sizes are large and reduces the variance, which is also attractive in practice.

(3) \emph{PD controller (Figure~\ref{fig:experiments-PD}).} A pure D controller ($\alpha_D = 1$) leads to poor performance. In our case, the difference of errors would be $g_t - g_{t-1} = b_{t-1}(x_{t-1}) - b_t(x_t)$. When expenditures are noisy and actions do not change too much from one period to the next, the signal-to-noise ratio of the difference in errors is usually too low to be useful in practice. When the step-size multiplier is low ($s = 0.1$), introducing a derivative controller hurts performance. The derivative term has little impact around the optimal step-size multiplier of $s=10$ (except for the case $\alpha_D = 1$). Interestingly, introducing a D term does increase performance when the step-size is large. A possible explanation for this behavior is that the D term, which corresponds to optimism, makes the algorithm an approximation of the classic proximal point method~\citep{mokhtari2020unified}. As an implicit discretization of gradient flow, it is well known that the proximal point method is more stable with large step-sizes than the explicit discretization of gradient flow, i.e., gradient descent~\citep{larsson2003partial,lu2022sr}.

(4) \emph{PID controller (Figure~\ref{fig:experiments-PID}).} With the best step-size multiplier choice (i.e., $s=10$) and momentum parameter (i.e., $\beta=0.99)$, it turns out that adding the I term and/or the D term does not significantly affect the performance of the algorithm. Finally, we would like to mention that while not all choices of parameters satisfy the real-roots assumption (Assumption \ref{ass:real-roots}), all of the parameter choices shown in Figure~\ref{fig:experiments-PID} achieve a high competitive ratio of around 0.9. As such, we conjecture that the asymptotic optimality of PID controllers may hold even without Assumption~\ref{ass:real-roots}.
    


An interesting takeaway is that some degree of momentum via the integral term can increase the robustness of the algorithm. In practice, it is usually impossible to pick the optimal step-size since this depends on the parameters of the problem, which are usually unknown, and even if the estimates are known, these can change unexpectedly in the real world. Therefore, one could end up overestimating or underestimating the true optimal step-size. Although momentum has little impact if step-sizes are underestimated, it can hedge against the risk of overestimating step-sizes, thus increasing the operating range in which the algorithm attains good performance.

\section{Conclusion and Future Directions}

{
In this paper, we uncover a fundamental connection between dual-based PID controllers and first-order algorithms for online allocation problems. We introduce CMD, a first-order algorithm for online convex optimization, which updates iterates based on a weighted moving average of past gradients and leverages a regret bound for CMD to give the first performance results for dual-based PID controllers for online allocation problems. 

We conclude by listing a few open questions and future directions. In many online allocation problems (taking budget pacing as an example), the data is usually neither adversarial nor i.i.d.~in practice but shares certain non-stationary patterns. It is interesting to study the performance of dual-based PID controllers in a non-stationary setting. 
Finally, it remains open to show whether CMD can be helpful for general online mirror descent in the stochastic i.i.d.~setting. While we show that CMD can lead to low regret, the regret bound deteriorates as the amount of momentum increases. We believe that this degradation is unavoidable in the case of adversarial input, and, in the case of stochastic input, it should be possible to leverage the statistical averaging effect of momentum to show that introducing momentum strictly improves regret.}

\bibliographystyle{plainnat}
\bibliography{references}

\begin{thebibliography}{54}
\providecommand{\natexlab}[1]{#1}
\providecommand{\url}[1]{\texttt{#1}}
\expandafter\ifx\csname urlstyle\endcsname\relax
  \providecommand{\doi}[1]{doi: #1}\else
  \providecommand{\doi}{doi: \begingroup \urlstyle{rm}\Url}\fi

\bibitem[Agrawal et~al.(2014)Agrawal, Wang, and Ye]{Agrawal2014OR}
Shipra Agrawal, Zizhuo Wang, and Yinyu Ye.
\newblock A dynamic near-optimal algorithm for online linear programming.
\newblock \emph{Operations Research}, 62\penalty0 (4):\penalty0 876--890, 2014.

\bibitem[Alacaoglu et~al.(2020)Alacaoglu, Malitsky, Mertikopoulos, and
  Cevher]{alacaoglu2020new}
Ahmet Alacaoglu, Yura Malitsky, Panayotis Mertikopoulos, and Volkan Cevher.
\newblock A new regret analysis for {A}dam-type algorithms.
\newblock In Hal~Daumé III and Aarti Singh, editors, \emph{Proceedings of the
  37th International Conference on Machine Learning}, volume 119 of
  \emph{Proceedings of Machine Learning Research}, pages 202--210. PMLR, 13--18
  Jul 2020.

\bibitem[{\AA}str{\"o}m et~al.(2006){\AA}str{\"o}m, H{\"a}gglund, and
  Astrom]{aastrom2006advanced}
Karl~Johan {\AA}str{\"o}m, Tore H{\"a}gglund, and Karl~J Astrom.
\newblock \emph{Advanced PID control}, volume 461.
\newblock ISA-The Instrumentation, Systems, and Automation Society Research
  Triangle Park, 2006.

\bibitem[Balseiro et~al.(2020)Balseiro, Lu, and Mirrokni]{balseiro2020dual}
Santiago Balseiro, Haihao Lu, and Vahab Mirrokni.
\newblock Dual mirror descent for online allocation problems.
\newblock In \emph{International Conference on Machine Learning}, pages
  613--628. PMLR, 2020.

\bibitem[Balseiro and Gur(2017)]{BalseiroGur2017EC}
Santiago~R. Balseiro and Yonatan Gur.
\newblock Learning in repeated auctions with budgets: Regret minimization and
  equilibrium.
\newblock In \emph{Proceedings of the 2017 ACM Conference on Economics and
  Computation}, EC ’17, page 609, New York, NY, USA, 2017. Association for
  Computing Machinery.
\newblock ISBN 9781450345279.

\bibitem[Balseiro and Gur(2019)]{BalseiroGur2019MS}
Santiago~R. Balseiro and Yonatan Gur.
\newblock Learning in repeated auctions with budgets: Regret minimization and
  equilibrium.
\newblock \emph{Management Science}, 65\penalty0 (9):\penalty0 3952--3968,
  2019.

\bibitem[Balseiro et~al.(2022)Balseiro, Lu, and Mirrokni]{balseiro2020best}
Santiago~R. Balseiro, Haihao Lu, and Vahab Mirrokni.
\newblock The best of many worlds: Dual mirror descent for online allocation
  problems.
\newblock \emph{Operations Research}, page to appear, 2022.

\bibitem[Bauschke et~al.(2001)Bauschke, Borwein, and
  Combettes]{bauschke2001essential}
Heinz~H Bauschke, Jonathan~M Borwein, and Patrick~L Combettes.
\newblock Essential smoothness, essential strict convexity, and legendre
  functions in banach spaces.
\newblock \emph{Communications in Contemporary Mathematics}, 3\penalty0
  (04):\penalty0 615--647, 2001.

\bibitem[Beck(2017)]{beck2017first}
Amir Beck.
\newblock \emph{First-order methods in optimization}.
\newblock SIAM, 2017.

\bibitem[Bennett(2001)]{bennett2001past}
Stuart Bennett.
\newblock The past of pid controllers.
\newblock \emph{Annual Reviews in Control}, 25:\penalty0 43--53, 2001.

\bibitem[Bernstein et~al.(2015)Bernstein, K{\"o}k, and
  Xie]{bernstein2015dynamic}
Fernando Bernstein, A~G{\"u}rhan K{\"o}k, and Lei Xie.
\newblock Dynamic assortment customization with limited inventories.
\newblock \emph{Manufacturing \& Service Operations Management}, 17\penalty0
  (4):\penalty0 538--553, 2015.

\bibitem[Bertsekas(1999)]{Bertsekas}
D.~Bertsekas.
\newblock \emph{Nonlinear Programming}.
\newblock Athena Scientific, Belmont, MA, 1999.

\bibitem[Bitran and Caldentey(2003)]{bitran2003overview}
Gabriel Bitran and Ren{\'e} Caldentey.
\newblock An overview of pricing models for revenue management.
\newblock \emph{Manufacturing \& Service Operations Management}, 5\penalty0
  (3):\penalty0 203--229, 2003.

\bibitem[Chen and Teboulle(1993)]{chen1993convergence}
Gong Chen and Marc Teboulle.
\newblock Convergence analysis of a proximal-like minimization algorithm using
  bregman functions.
\newblock \emph{SIAM Journal on Optimization}, 3\penalty0 (3):\penalty0
  538--543, 1993.

\bibitem[Chen et~al.(2018)Chen, Zhou, Tang, Yang, Cao, and Gu]{chen2018closing}
Jinghui Chen, Dongruo Zhou, Yiqi Tang, Ziyan Yang, Yuan Cao, and Quanquan Gu.
\newblock Closing the generalization gap of adaptive gradient methods in
  training deep neural networks.
\newblock \emph{arXiv preprint arXiv:1806.06763}, 2018.

\bibitem[Cominos and Munro(2002)]{cominos2002pid}
P~Cominos and N~Munro.
\newblock Pid controllers: recent tuning methods and design to specification.
\newblock \emph{IEE Proceedings-Control Theory and Applications}, 149\penalty0
  (1):\penalty0 46--53, 2002.

\bibitem[Conitzer et~al.(2021)Conitzer, Kroer, Sodomka, and
  Stier-Moses]{conitzer2021multiplicative}
Vincent Conitzer, Christian Kroer, Eric Sodomka, and Nicolas~E Stier-Moses.
\newblock Multiplicative pacing equilibria in auction markets.
\newblock \emph{Operations Research}, 2021.

\bibitem[Daskalakis and Panageas(2018)]{daskalakis2018limit}
Constantinos Daskalakis and Ioannis Panageas.
\newblock The limit points of (optimistic) gradient descent in min-max
  optimization.
\newblock \emph{Advances in neural information processing systems}, 31, 2018.

\bibitem[Daskalakis et~al.(2017)Daskalakis, Ilyas, Syrgkanis, and
  Zeng]{daskalakis2017training}
Constantinos Daskalakis, Andrew Ilyas, Vasilis Syrgkanis, and Haoyang Zeng.
\newblock Training gans with optimism.
\newblock \emph{arXiv preprint arXiv:1711.00141}, 2017.

\bibitem[Devanur and Hayes(2009)]{DevanurHayes2009}
Nikhil~R Devanur and Thomas~P Hayes.
\newblock The adwords problem: online keyword matching with budgeted bidders
  under random permutations.
\newblock In \emph{Proceedings of the 10th ACM conference on Electronic
  commerce}, pages 71--78, 2009.

\bibitem[Devanur et~al.(2019)Devanur, Jain, Sivan, and
  Wilkens]{Devanur2019near}
Nikhil~R Devanur, Kamal Jain, Balasubramanian Sivan, and Christopher~A Wilkens.
\newblock Near optimal online algorithms and fast approximation algorithms for
  resource allocation problems.
\newblock \emph{Journal of the ACM (JACM)}, 66\penalty0 (1):\penalty0 1--41,
  2019.

\bibitem[Duchi et~al.(2011)Duchi, Hazan, and Singer]{duchi2011adaptive}
John Duchi, Elad Hazan, and Yoram Singer.
\newblock Adaptive subgradient methods for online learning and stochastic
  optimization.
\newblock \emph{Journal of machine learning research}, 12\penalty0 (7), 2011.

\bibitem[Feldman et~al.(2009)Feldman, Mehta, Mirrokni, and
  Muthukrishnan]{feldman2009online}
Jon Feldman, Aranyak Mehta, Vahab Mirrokni, and Shan Muthukrishnan.
\newblock Online stochastic matching: Beating 1-1/e.
\newblock In \emph{50th Annual IEEE Symposium on Foundations of Computer
  Science}, pages 117--126, 2009.

\bibitem[Feldman et~al.(2010)Feldman, Henzinger, Korula, Mirrokni, and
  Stein]{Feldman2010}
Jon Feldman, Monika Henzinger, Nitish Korula, Vahab~S. Mirrokni, and Cliff
  Stein.
\newblock Online stochastic packing applied to display ad allocation.
\newblock In \emph{Proceedings of the 18th annual European conference on
  Algorithms}, 2010.

\bibitem[Gallego et~al.(2019)Gallego, Topaloglu, et~al.]{gallego2019revenue}
Guillermo Gallego, Huseyin Topaloglu, et~al.
\newblock \emph{Revenue management and pricing analytics}, volume 209.
\newblock Springer, 2019.

\bibitem[Golrezaei et~al.(2014)Golrezaei, Nazerzadeh, and
  Rusmevichientong]{golrezaei2014real}
Negin Golrezaei, Hamid Nazerzadeh, and Paat Rusmevichientong.
\newblock Real-time optimization of personalized assortments.
\newblock \emph{Management Science}, 60\penalty0 (6):\penalty0 1532--1551,
  2014.

\bibitem[Gummadi et~al.(2012)Gummadi, Key, and Proutiere]{gummadi2012repeated}
Ramakrishna Gummadi, Peter Key, and Alexandre Proutiere.
\newblock Repeated auctions under budget constraints: Optimal bidding
  strategies and equilibria.
\newblock In \emph{the Eighth Ad Auction Workshop}, 2012.

\bibitem[Gupta and Molinaro(2016)]{gupta2016experts}
Anupam Gupta and Marco Molinaro.
\newblock How the experts algorithm can help solve lps online.
\newblock \emph{Mathematics of Operations Research}, 41\penalty0 (4):\penalty0
  1404--1431, 2016.

\bibitem[Hazan et~al.(2016)]{hazan2016introduction}
Elad Hazan et~al.
\newblock Introduction to online convex optimization.
\newblock \emph{Foundations and Trends{\textregistered} in Optimization},
  2\penalty0 (3-4):\penalty0 157--325, 2016.

\bibitem[Huang et~al.(2018)Huang, Wang, and Dong]{huang2018nostalgic}
Haiwen Huang, Chang Wang, and Bin Dong.
\newblock Nostalgic adam: Weighting more of the past gradients when designing
  the adaptive learning rate.
\newblock \emph{arXiv preprint arXiv:1805.07557}, 2018.

\bibitem[Kakade et~al.(2009)Kakade, Shalev-Shwartz, and
  Tewari]{kakade2009duality}
Sham Kakade, Shai Shalev-Shwartz, and Ambuj Tewari.
\newblock On the duality of strong convexity and strong smoothness: Learning
  applications and matrix regularization.
\newblock 2009.

\bibitem[Karp et~al.(1990)Karp, Vazirani, and Vazirani]{karp1990optimal}
Richard~M Karp, Umesh~V Vazirani, and Vijay~V Vazirani.
\newblock An optimal algorithm for on-line bipartite matching.
\newblock In \emph{Proceedings of the twenty-second annual ACM symposium on
  Theory of computing}, pages 352--358, 1990.

\bibitem[Kesselheim et~al.(2014)Kesselheim, T{\"o}nnis, Radke, and
  V{\"o}cking]{kesselheim2014primal}
Thomas Kesselheim, Andreas T{\"o}nnis, Klaus Radke, and Berthold V{\"o}cking.
\newblock Primal beats dual on online packing lps in the random-order model.
\newblock In \emph{Proceedings of the forty-sixth annual ACM symposium on
  Theory of computing}, pages 303--312, 2014.

\bibitem[Kingma and Ba(2015)]{kingma2015adam}
Diederik~P. Kingma and Jimmy Ba.
\newblock Adam: {A} method for stochastic optimization.
\newblock In \emph{3rd International Conference on Learning Representations},
  2015.

\bibitem[Larsson and Thom{\'e}e(2003)]{larsson2003partial}
Stig Larsson and Vidar Thom{\'e}e.
\newblock \emph{Partial differential equations with numerical methods},
  volume~45.
\newblock Springer, 2003.

\bibitem[Li et~al.(2020)Li, Sun, and Ye]{li2020simple}
Xiaocheng Li, Chunlin Sun, and Yinyu Ye.
\newblock Simple and fast algorithm for binary integer and online linear
  programming.
\newblock \emph{arXiv preprint arXiv:2003.02513}, 2020.

\bibitem[Lu(2022)]{lu2022sr}
Haihao Lu.
\newblock An o (sr)-resolution ode framework for understanding discrete-time
  algorithms and applications to the linear convergence of minimax problems.
\newblock \emph{Mathematical Programming}, 194\penalty0 (1-2):\penalty0
  1061--1112, 2022.

\bibitem[Mokhtari et~al.(2020)Mokhtari, Ozdaglar, and
  Pattathil]{mokhtari2020unified}
Aryan Mokhtari, Asuman Ozdaglar, and Sarath Pattathil.
\newblock A unified analysis of extra-gradient and optimistic gradient methods
  for saddle point problems: Proximal point approach.
\newblock In \emph{International Conference on Artificial Intelligence and
  Statistics}, pages 1497--1507. PMLR, 2020.

\bibitem[Nesterov(2003)]{nesterovBook}
Y.~Nesterov.
\newblock \emph{Introductory lectures on convex optimization: a basic course}.
\newblock Kluwer Academic Publishers, Boston, 2003.

\bibitem[Nesterov(1983)]{nesterov1983method}
Yurii Nesterov.
\newblock A method of solving a convex programming problem with convergence
  rate ${O}(1/k^2)$.
\newblock In \emph{Soviet Mathematics Doklady}, volume~27, pages 372--376,
  1983.

\bibitem[Nielsen et~al.(2007)Nielsen, Boissonnat, and Nock]{nielsen2007bregman}
Frank Nielsen, Jean-Daniel Boissonnat, and Richard Nock.
\newblock Bregman voronoi diagrams: Properties, algorithms and applications.
\newblock \emph{arXiv preprint arXiv:0709.2196}, 2007.

\bibitem[Niven(1969)]{niven1969formal}
Ivan Niven.
\newblock Formal power series.
\newblock \emph{The American Mathematical Monthly}, 76\penalty0 (8):\penalty0
  871--889, 1969.

\bibitem[Panda(2012)]{panda2012introduction}
Rames~C Panda.
\newblock \emph{Introduction to PID controllers: theory, tuning and application
  to frontier areas}.
\newblock BoD--Books on Demand, 2012.

\bibitem[Polyak(1964)]{polyak1964some}
Boris~T Polyak.
\newblock Some methods of speeding up the convergence of iteration methods.
\newblock \emph{Ussr computational mathematics and mathematical physics},
  4\penalty0 (5):\penalty0 1--17, 1964.

\bibitem[Reddi et~al.(2019)Reddi, Kale, and Kumar]{reddi2019convergence}
Sashank~J Reddi, Satyen Kale, and Sanjiv Kumar.
\newblock On the convergence of adam and beyond.
\newblock \emph{arXiv preprint arXiv:1904.09237}, 2019.

\bibitem[Smirnov et~al.(2016)Smirnov, Lu, and Lee]{smirnov2016online}
Yury Smirnov, Quan Lu, and Kuang-chih Lee.
\newblock Online ad campaign tuning with pid controllers, April~21 2016.
\newblock US Patent App. 14/518,601.

\bibitem[Sra et~al.(2012)Sra, Nowozin, and Wright]{sra2012optimization}
Suvrit Sra, Sebastian Nowozin, and Stephen~J Wright.
\newblock \emph{Optimization for machine learning}.
\newblock Mit Press, 2012.

\bibitem[Sutskever et~al.(2013)Sutskever, Martens, Dahl, and
  Hinton]{sutskever2013importance}
Ilya Sutskever, James Martens, George Dahl, and Geoffrey Hinton.
\newblock On the importance of initialization and momentum in deep learning.
\newblock In \emph{International conference on machine learning}, pages
  1139--1147. PMLR, 2013.

\bibitem[Talluri and Van~Ryzin(2006)]{talluri2006theory}
Kalyan~T Talluri and Garrett~J Van~Ryzin.
\newblock \emph{The theory and practice of revenue management}, volume~68.
\newblock Springer Science \& Business Media, 2006.

\bibitem[Tashman et~al.(2020)Tashman, Xie, Hoffman, Winikor, and
  Gerami]{tashman2020dynamic}
Michael Tashman, Jiayi Xie, John Hoffman, Lee Winikor, and Rouzbeh Gerami.
\newblock Dynamic bidding strategies with multivariate feedback control for
  multiple goals in display advertising.
\newblock \emph{arXiv preprint arXiv:2007.00426}, 2020.

\bibitem[Vinagre et~al.(2007)Vinagre, Monje, Calder{\'o}n, and
  Su{\'a}rez]{vinagre2007fractional}
Blas~M Vinagre, Concepci{\'o}n~A Monje, Antonio~J Calder{\'o}n, and Jos{\'e}~I
  Su{\'a}rez.
\newblock Fractional pid controllers for industry application. a brief
  introduction.
\newblock \emph{Journal of Vibration and Control}, 13\penalty0 (9-10):\penalty0
  1419--1429, 2007.

\bibitem[Yang et~al.(2019)Yang, Li, Wang, Wu, Tan, Xu, and Gai]{yang2019bid}
Xun Yang, Yasong Li, Hao Wang, Di~Wu, Qing Tan, Jian Xu, and Kun Gai.
\newblock Bid optimization by multivariable control in display advertising.
\newblock In \emph{Proceedings of the 25th ACM SIGKDD International Conference
  on Knowledge Discovery \& Data Mining}, pages 1966--1974, 2019.

\bibitem[Ye et~al.(2020)Ye, Zhang, Zhang, Zhang, Chen, and Xu]{ye2020cold}
Zikun Ye, Dennis Zhang, Heng Zhang, Renyu~Philip Zhang, Xin Chen, and Zhiwei
  Xu.
\newblock Cold start to improve market thickness on online advertising
  platforms: Data-driven algorithms and field experiments.
\newblock \emph{Available at SSRN 3702786}, 2020.

\bibitem[Zhang et~al.(2016)Zhang, Rong, Wang, Zhu, and Wang]{zhang2016feedback}
Weinan Zhang, Yifei Rong, Jun Wang, Tianchi Zhu, and Xiaofan Wang.
\newblock Feedback control of real-time display advertising.
\newblock In \emph{Proceedings of the Ninth ACM International Conference on Web
  Search and Data Mining}, pages 407--416, 2016.

\end{thebibliography}

\newpage

\appendix

\newpage
\appendix

\section{Missing Proofs of the Main Results}
\subsection{Proof of Proposition \ref{prop:PID-equi}}\label{sec:proof:prop:PID-equi}
Let $e_t = (1-\beta) \sum_{s=0}^{t-1} \beta^{s} g_{t-s}$. First, it is easy to check from the definitions of $\lambda_i$, $\alpha_P,\alpha_I$, and $\alpha_D$ that 
\begin{align*}
    \eta z_t &= \eta \sum_{s\le t} \lambda_{t-s} g_s = \eta \lambda_0 g_t + \eta \lambda_1 g_{t-1} + \eta \sum_{s\le t-2} \lambda_{t-s} g_s\\
    &=\eta \alpha_P g_t + \eta \alpha_I e_t + \eta \alpha_D (g_t - g_{t-1}) = K_P g_t + K_I \sum_{s=0}^{t-1} \beta^{s} g_{t-s} + K_D(g_t - g_{t-1})\ .
\end{align*}
Furthermore, notice that the last of the three expressions in~\eqref{eq:new-al} is a mirror descent step, and we can rewrite it as
\begin{equation}
    \nabla h(\tmu_t)=\nabla h(\mu_t)-\eta z_t\ , \ \ \ \  \mu_{t+1}=\arg\min_{\mu\in \mathbb R_+^m} V_h(\mu,\tmu_t) \,.
\end{equation}
The solution $\mu_t$ is first mapped to the ``dual'' space using the gradient of the reference function to a solution $y_t = \nabla h(\mu_t)$, i.e., $(y_t)_j=\ell_j(\mu_{t,j})$, then the dual solution is updated in the direction of the average gradient to obtain a new solution $\tilde y_t = \nabla h(\tilde \mu_t)$, and solutions are finally projected back to the feasible set $\mathbb R_+^m$ by $\ell_j^{-1}$. Since $h$ is coordinate-wise separable, we recover \eqref{eq:non-linear-update}. This finishes the proof.

\subsection{Proofs of Lemma~\ref{lemma:diff-iterate}, Lemma~\ref{lem:decomposition}, and Lemma~\ref{lemma:omd}}\label{sec:proof_thm}

    

\begin{proof}[Proof of Lemma \ref{lemma:diff-iterate}]
The dual mirror descent update in Algorithm~\ref{al:omd-m} can be written as 
\begin{equation}\label{eq:another_update}
    \nabla h(\tmu_t)=\nabla h(\mu_t)-\eta z_t\ , \ \ \ \  \mu_{t+1}=\arg\min_{\mu\in\mathcal U} V_h(\mu,\tmu_t) \ ,
\end{equation}
which can be easily verified by noticing that the optimality condition of Algorithm~\ref{al:omd-m} and \eqref{eq:another_update} are both
$0\in \nabla h(\mu_{t+1})-\nabla h(\mu_t) +\eta z_t + \mathcal N_{\mathcal U}(\mu_{t+1})$, where { $\mathcal N_{\mathcal U}(\mu_{t+1})=\{g\in\RR^m | g^\top(\mu-\mu_{t+1})\le 0 \text{ for all } \mu\in\mathcal U\}$} is the normal cone of $\mathcal U$ at $\mu_{t+1}$. { Let $h^*$ be the convex conjugate of $h$, then $\nabla h^*=(\nabla h)^{-1}$, and $\nabla h^*$ is $1/\sigma$-Lipschitz continuous in the primal norm by recalling that $h$ is $\sigma$-strongly-convex in the dual norm (see, e.g.,} \cite{kakade2009duality}). Thus,
\begin{align}\label{eq:diff-mu}
\begin{split}
        \|\tmu_t-\mu_t\dualnorm=&\|\nabla h^*(\nabla h (\tmu_{t})) - \mu_t\dualnorm=\|\nabla h^*(\nabla h (\mu_{t})-\eta z_t) - \mu_t\dualnorm\\
    =&\|\nabla h^*(\nabla h (\mu_{t})-\eta z_t) - \nabla h^*(\nabla h (\mu_{t}))\dualnorm\le \frac{\eta}{\sigma} \|z_t\primalnorm \ .
\end{split}
\end{align}
Meanwhile, it follows by the generalized Pythagorean Theorem of Bregman projection~\cite{nielsen2007bregman} that
\begin{equation}\label{eq:m-1}
    V_h(\mu_t, \tmu_t)\ge V_h(\mu_t, \mu_{t+1})+V_h(\mu_{t+1},\tmu_t) \ge V_h(\mu_t, \mu_{t+1}) \ge \frac{\sigma}{2}\|\mu_{t+1}-\mu_t\dualnorm^2 \ ,
\end{equation}
where the second inequality is from the non-negativity of Bregman divergence and the last inequality uses the strong-convexity of $h$ w.r.t.~dual norm. On the other hand, we have
\begin{align}\label{eq:m-2}
\begin{split}
    V_h(\mu_t, \tmu_t)&\le V_h(\mu_t, \tmu_t)+V_h(\tmu_t,\mu_t)= \left(\nabla h (\tmu_t)-\nabla h(\mu_t) \right)^\top \left(\tmu_t-\mu_t\rangle\right)\\
    &\le \|\nabla h (\tmu_t)-\nabla h(\mu_t)\| \|\tmu_t-\mu_t\dualnorm \le \frac{\eta^2}{\sigma}\|z_t\primalnorm^2\ ,
\end{split}
\end{align}
where the second inequality follows from Cauchy-Schwartz, and the third inequality utilizes \eqref{eq:another_update} and \eqref{eq:diff-mu}. 
We finish the proof by combining \eqref{eq:m-1} and \eqref{eq:m-2}.
\end{proof}

\begin{proof}[Proof of Lemma \ref{lem:decomposition}]
    We just need to check the corresponding coefficients with respect to $g_t \mu_t$ and $g_t \mu$ of both sides of the equation \eqref{eq:decomposition} are the same. 

    { Let $\vec{g} = (g_t^\top)_{t=1}^T \in \mathbb R^{T\times m}$ be a matrix with gradients stacked vertically and $\vec{z} = (z_t^\top )_{t=1}^T \in \mathbb R^{T\times m}$ be a matrix with the convolution of gradients stacked vertically. We have that $\vec{z} = R \vec{g}$ and $\vec{g} = Q \vec{z}$. Moreover, denote by $\vec{a} = (a_t)_{t=1}^T$ the auxiliary sequence defined in the statement of the theorem, which can be written as $\vec{a}^\top = \mathbf 1^\top Q$  with $\mathbf 1$ a vector of ones with $T$ entries, i.e., $a$ is the sum of the columns of $Q$.
    Notice that in matrix notation, we can write for all $\mu \in \mathbb R^m$
    \begin{align*}
        \sum_{t=1}^T a_t z_t^\top \mu &= \vec{a}^\top \vec{z} \mu = \mathbf 1^\top Q R \vec{g} \mu = \mathbf 1^\top \vec{g} \mu = \sum_{t=1}^T g_t^{\top} \mu\,,
    \end{align*}
    where we used that $Q = R^{-1}$. Thus, the coefficient of $g_t \mu$ in both sides of \eqref{eq:decomposition} matches. Moreover, we have the following identity for $t<T$ 
    \begin{equation}\label{eq:sum}
        \sum_{t\ge s} a_t \lambda_{t-s} = \vec{a}^\top R e_t = \mathbf 1^\top Q R e_t =\mathbf 1^\top e_t = 1\,,
    \end{equation}
    where we denote by $e_t$ the $T$-dimensional unit vector and $Q= R^{-1}$.}
    
    
    We next check that the coefficients of $g_t \mu_t$ match. We have 
    \begin{align*}
        \sum_{t=1}^T g_t^{\top} \mu_t - \sum_{t=1}^T a_t z_t^\top \mu_t &= \sum_{t=1}^T g_t^{\top} \mu_t - \sum_{s=1}^T \pran{\sum_{t\ge s} a_t \lambda_{t-s}} g_s^{\top} \mu_t = \sum_{t=1}^T g_t^{\top} \mu_t -\sum_{t=1}^T g_t^{\top} \sum_{s=t}^T a_s \lambda_{s-t}  \mu_s\\
        &= \sum_{t=1}^T \pran{1-\sum_{s=t}^T a_s \lambda_{s-t}} g_t^{\top} \mu_t -\sum_{t=1}^T g_t^{\top} \sum_{j=t+1}^T (\mu_{j}-\mu_{j-1}) \sum_{s=j}^T a_s \lambda_{s-t}\\
        &=-\sum_{t=1}^T g_t^{\top} \sum_{j=t+1}^T (\mu_{j}-\mu_{j-1}) \sum_{s=j}^T a_s \lambda_{s-t} = -\sum_{s=1}^T \sum_{t=s+1}^T b_{t,s} g_s^{\top} (\mu_t-\mu_{t-1})\ ,
    \end{align*}
    where the third equality uses     that $\mu_s=\mu_t+\sum_{j=t+1}^s (\mu_j-\mu_{j-1})$, and the fourth uses \eqref{eq:sum}. This finishes the proof.    
\end{proof}

\begin{proof}[Proof of Lemma~\ref{lemma:omd}]
Because the Bregman divergence $V_h$ is differentiable and convex in its first argument, and $\mathcal U$ is convex, the first order conditions for the Bregman projection (see, e.g., Proposition 2.1.2 in \cite{Bertsekas}) are given by
\begin{align}\label{eq:foc-bregman}
    \left( z_t + \frac 1 \eta \left( \nabla h(\mu_{t+1}) - \nabla h(\mu_t)\right)\right)^\top \left( \mu - \mu_{t+1} \right) \ge 0\,, \quad \forall \mu \in \mathcal U\,.
\end{align}
Therefore, it holds for any $\mu \in \mathcal U$ that
\begin{equation}  \label{eq:ineq_chain1}
\begin{split}
     \langle z_t, \mu_t - \mu \rangle\
    &=\langle z_t, \mu_t - \mu_{t+1} \rangle\ + \langle z_t, \mu_{t+1} - \mu \rangle\ \\
    &\le\langle z_t, \mu_t - \mu_{t+1} \rangle\ + \frac 1 \eta \left( \nabla h(\mu_{t+1}) - \nabla h(\mu_t)\right)^\top \left( \mu - \mu_{t+1} \right) \\
    &=\langle z_t, \mu_t - \mu_{t+1}\rangle + \frac{1}{\eta} V_h(\mu,\mu_t) - \frac{1}{\eta} V_h(\mu,\mu_{t+1}) - \frac{1}{\eta} V_h(\mu_{t+1},\mu_t) \\
 &\le  \langle z_t, \mu_t - \mu_{t+1}\rangle + \frac{1}{\eta} V_h(\mu,\mu_t) - \frac{1}{\eta} V_h(\mu,\mu_{t+1}) - \frac{\sigma}{2\eta} \|\mu_{t+1}-\mu_t\dualnorm^2 \\
  &\le   \frac{\eta}{2\sigma}\|z_t\primalnorm^2  + \frac{1}{\eta} V_h(\mu,\mu_t) - \frac{1}{\eta} V_h(\mu,\mu_{t+1}) \ , 
\end{split}
\end{equation}
where the first inequality follows from \eqref{eq:foc-bregman}; the second equality follows from Three-Point Property stated in Lemma 3.1 of \cite{chen1993convergence}; the second inequality is by strong convexity of $h$; and the third inequality uses $a^2 + b^2 \ge 2 a b$ for $a,b \in \RR$ and Cauchy-Schwartz to obtain
\begin{align*}
    \frac{\sigma}{2\eta} \|\mu_{t+1}-\mu_t\dualnorm^2+\frac{\eta}{2\sigma}\|z_t\primalnorm^2 &\ge \|\mu_{t+1}-\mu_t\dualnorm\|z_t\primalnorm \ge   |\langle z_t, \mu_t - \mu_{t+1}\rangle| \ ,
\end{align*}
The proof follows.
\end{proof}

\section{Proof of Lemma \ref{lemma:iterates}}\label{app:missing}

\begin{proof}
    
Fix a resource $j$. Denoting $\bar G= \rho_j + \ubbinfty$, we observe that gradients satisfy $|g_{t,j}| \le |b_{t,j}(x_t)| + |B_j/T| \le \ubbinfty + \rho_j = \bar G$. We assume that $B_j/T < \ubbinfty$ as otherwise, the dual variables are monotonically decreasing and the bound is trivial. Using separability of the reference function, we can denote by  $\Delta_{t,j} = \ddh_j(\mu_{t,j}) - \ddh_j(\mu_{t-1,j})$ the difference of the iterates evaluated at the derivative of the reference function for resource $j$.

We first argue that $|\Delta_{t,j}| \le 2 \eta \bar G$ for all $t \ge 1$. We can write the update as follows: $\ddh_j(\tmu_{t,j}) = \ddh_j(\mu_{t,j}) - \eta z_{t,j}$ and $\mu_{t+1,j} = \max(\tmu_{t,j},0)$. Let $e_{t,j} = (1-\beta) \sum_{s=1}^t \beta^{t-s} g_{s,j}$. Therefore, 
\begin{align}\label{eq:bound-squared-change}
    |\Delta_{t+1,j}| &=  \left|  \ddh_j(\mu_{t,j}) - \ddh_j(\mu_{t-1,j}) \right| \le \left|\ddh_j(\tmu_{t,j}) - \ddh_j(\mu_{t-1,j}) \right| = \eta |z_{t,j} |\nonumber\\
    &= \eta \left| \alpha_P g_{t,j} + \alpha_I e_{t,j} + \alpha_D (g_{t,j} - g_{t-1,j} \right|\nonumber\\
    &\le \eta \left[ \alpha_P | g_{t,j}| + \alpha_I |e_{t,j}| + \alpha_D( | g_{t,j}| + | g_{t-1,j}| )\right]
    \le 2 \eta \bar G\,,
\end{align}
where the first inequality follows because  $\ddh_j$ is monotone, the second inequality from the triangle inequality, and the last inequality because $\alpha_P + \alpha_I + \alpha_D = 1$ together with $|e_{t,j}| \le (1-\beta) \sum_{s=1}^t \beta^{t-s} |g_{s,j}| \le \bar G (1-\beta) \sum_{s=1}^t \beta^{t-s} = \bar G (1-\beta^t)/(1-\beta) \le \bar G$ from the formula for the geometric sum and because $\beta \in (0,1)$.

Define $h^*_j(c)=\max_{\mu_j} \{c \mu_j - h_j(\mu_j)\}$ as the conjugate function of $h_j(\mu_j)$, then by Assumption~\ref{ass:h-sep} it holds that $h^*_j(\cdot)$ is a $\frac{1}{\sigma}$-smooth univariate convex function~\citep{kakade2009duality}. Furthermore, $\ddh^*_j(\cdot)$ is increasing, and $\ddh^*_j(\ddh_j(\mu))=\mu$. 

We show the result by contradiction. Suppose there exists a time period $\bar{\tau}$ and a resource $j$ such that $\mu_{\bar{\tau},j}>\mu_j^{\max}>\frac{\bar{f}}{\rho_j}$. Then there must be a time period $1\le \underline \tau\le \bar{\tau}$ such that $\mu_{\underline \tau -1,j} < \bar f / \rho_j < \mu_{\underline \tau,j},\mu_{\underline \tau+1,j},..., \mu_{\bar{\tau},j}$ by noticing $\mu_{1,j} \le \bar f / \rho_j \le \mumax$. Now consider a sequence of consecutive time periods $t=\underline \tau, \tau+1,\ldots,\bar{\tau}-1$. We know that there must never be projection to the positive orthant in these steps, because otherwise $\mu_{t+1,j}=0$, which violates our definition of the sequence of time periods $\underline \tau,\ldots,\bar{\tau}-1$.

Then, we have for all $t = \underline \tau,\ldots,\bar \tau-1$
\begin{align*}
 \Delta_{t+1,j} &= \ddh_j(\mu_{t+1,j}) - \ddh_j(\mu_{t,j}) = - \eta z_{t,j} = -\eta \left[\alpha_P g_{t,j} + \alpha_I e_{t,j}+\alpha_D ( g_{t,j} - g_{t-1,j} ) \right] \\
    &=-\eta\left[\alpha_P g_{t,j} + \alpha_I (\beta e_{t-1,j}+(1-\beta)g_{t,j}) +\alpha_D ( g_{t,j} - g_{t-1,j}) \right]\,,
\end{align*}
where the second equation follows because the dual variables are not projected, the third from the definitions of $z_t$, and the fourth because $e_t = \beta e_{t-1} + (1-\beta) g_t$. Similarly, we have
\begin{align*}
 \beta \Delta_{t,j} &= - \beta \eta z_{t-1,j} = - \beta \eta \left(\alpha_P g_{t-1,j} + \alpha_I e_{t-1,j}+\alpha_D ( g_{t-1,j} - g_{t-2,j} ) \right)\,.
\end{align*}
We substract these two equations to cancel the term $e_{t-1,j}$ and obtain that
\begin{align}\label{eq:update-no-projection}
    \Delta_{t+1,j} &= \beta \Delta_{t,j} + \eta \alpha_P (\beta g_{t-1,j} - g_{t,j}) + \eta \alpha_D \left[ \beta (g_{t-1,j} - g_{t-2,j}) - (g_{t,j} - g_{t-1,j)} \right] - \eta \alpha_I (1-\beta) g_{t,j}\,.
\end{align}

Because rewards are bounded by $f_t(x) \le \ubf$, using that $0\in\cX$ is feasible and the choice of $x_t$, it holds that $0=f_t(0)\le f_t(x_t)-\mu_t^\top b_t(x_t)\le \ubf-\mu_t^\top b_t(x_t)$, whereby $\mu_t^\top b_t(x_t)\le \ubf$. Since $\mu_t\ge 0, b_t(x)\ge 0, x_t\in\cX\subseteq \RR^d_{+}$, it holds for every resource $j$ that $b_{t,j}(x_t) \le \frac{\ubf}{\mu_{t,j}}$. 
This implies that
\[
    b_{t,j}(x_t) 
    \le \frac{\ubf}{\mu_{t,j}} \le \rho_j\,,
\]    
because $\mu_{t,j} \ge \bar f / \rho_j$ for $t = \underline \tau,\ldots,\bar \tau$. Therefore, $-g_{t,j} = b_{t,j}(x_t) - \rho_j \le 0$ and \eqref{eq:update-no-projection} reduces to
\begin{align*}
 \Delta_{t+1,j} &\le 
    \beta \Delta_{t,j} +\eta\alpha_P\left( \beta g_{t-1,j}-g_{t,j} \right) + \eta \alpha_D \left[ \beta (g_{t-1,j} - g_{t-2,j}) - (g_{t,j} - g_{t-1,j)} \right]\,,    
\end{align*}
which implies, by recursively applying the inequality, that for $\underline \tau < t \le \bar \tau$
\begin{align}\label{eq:upper-bound-delta}
    \Delta_{t,j} &\le \beta^{t-\underline \tau} \Delta_{\underline \tau,j} + \eta \alpha_P \left( \beta^{t - \underline \tau} g_{\underline \tau - 1, j} - g_{t - 1,j} \right) + \eta \alpha_D \left[ \beta^{t - \underline \tau} \left( g_{\underline \tau - 1, j} - g_{\underline \tau - 2, j} \right) - \left(g_{t-1,j} - g_{t-2,j}\right) \right]\nonumber\\
    &\le \beta^{t-\underline \tau} \Delta_{\underline \tau,j} + \eta \alpha_P \beta^{t - \underline \tau} \bar G + 2 \eta \alpha_D \beta^{t - \underline \tau} \bar G - \eta \alpha_D \left(g_{t-1,j} - g_{t-2,j}\right)\,,
\end{align}
where the last inequality follows from using again that $-g_{t - 1,j} \ge 0$ and that $g_{t,j} = \rho_j - b_{t,j}(x_t) \le \rho_j$ because $b_{t,j}(x_t) \ge 0$. Therefore, for $\underline \tau \le t \le \bar \tau$ we have that
\begin{align*}
    \ddh_j(\mu_{t,j}) &= \ddh_j(\mu_{\underline \tau - 1,j})
    +\Delta_{\underline \tau,j} + \sum_{s=\underline \tau+1}^t \Delta_{s,j}\\
    &\le \ddh_j(\mu_{\underline \tau - 1,j})
    +  \Delta_{\underline \tau,j} \sum_{s=\underline \tau}^t \beta^{t-\underline \tau} + \eta \alpha_P \bar G \sum_{s=\underline \tau+1}^t \beta^{t-\underline \tau}+ 2 \eta \alpha_D \bar G \sum_{s=\underline \tau+1}^t \beta^{t-\underline \tau}
    + \eta \alpha_D \left(g_{\underline \tau - 1, j} - g_{t-1,j}     \right)\\
    &\le \ddh_j(\mu_{\underline \tau - 1,j})
    +  \Delta_{\underline \tau,j} \sum_{s=\underline \tau}^t \beta^{t-\underline \tau} + \eta \alpha_P \bar G \sum_{s=\underline \tau+1}^t \beta^{t-\underline \tau}+ 2 \eta \alpha_D \bar G \sum_{s=\underline \tau}^t \beta^{t-\underline \tau}\\
    &\le \ddh_j(\mu_{\underline \tau - 1,j}) + \frac{2 \eta \bar G + \eta \alpha_P \bar G + 2 \eta \alpha_D \bar G} {1-\beta}
    \le \ddh_j(\mu_{\underline \tau - 1,j}) + \frac{4 \eta \bar G} {1-\beta} \,,
\end{align*}
where first inequality follows from \eqref{eq:upper-bound-delta} and telescoping the derivative term, the second inequality because $|g_{\underline \tau-1,j}| \le \bar G$ and $g_{t-1,j} \ge 0$, the third inequality follows from \eqref{eq:bound-squared-change} together with $\sum_{s=\underline \tau}^t \beta^{t-\underline \tau} \le 1/(1-\beta)$, and the last because $\alpha_P + \alpha_D \le 1$. This implies that
\begin{align*}
    \mu_{\bar{\tau},j} &= \ddh^*_j(\ddh_j(\mu_{\bar{\tau},j}))
    \le \ddh^*_j\left( \ddh_j(\mu_{\underline \tau - 1,j}) + \frac{4 \eta \bar G} {1-\beta} \right)\le \ddh^*_j\left( \ddh_j(\mu_{\underline \tau - 1,j})\right) + \frac{4 \eta \bar G} {(1-\beta)\sigma}\\
    &   =
    \mu_{\underline \tau - 1,j} + \frac{4 \eta \bar G} {(1-\beta) \sigma} \le \mumax_j\,,
\end{align*}
where the first inequality follows from the monotonicity of $\ddh^*_j(\cdot)$, the second inequality is from the $\frac{1}{\sigma}$-smoothness of $h^*_j(\cdot)$, the second equality uses $\ddh^*_j(\ddh_j(\mu_j))=\mu_j$, and the last inequality utilizes $\mu_{\underline \tau - 1,j} \le \ubf / \rho_j$ and the definition of $\mumax_j$. The result follows by contradiction. 
\end{proof}

\end{document}